\documentclass[11pt]{amsart}
\usepackage{latexsym}
\usepackage{fullpage}
\usepackage{amssymb}
\usepackage{amscd}
\usepackage{float}
\usepackage{graphicx}
\usepackage{amsfonts}
\usepackage{pb-diagram}
 \usepackage{amsmath,amscd}
\usepackage{xcolor}

\newtheorem{thm}{Theorem}
\newtheorem{cor}{Corollary}
\newtheorem{lemma}{Lemma}
\newtheorem{prop}{Proposition}
\newtheorem{defn}{Definition}
\newtheorem{remark}{Remark}
\newtheorem{conj}{Conjecture}

\newtheorem{nt}{Notation}

\begin{document}
\title[HOMFLYPT skein sub-modules of the lens spaces $L(p, 1)$]
  {HOMFLYPT skein sub-modules of the lens spaces $L(p, 1)$}

\author{Ioannis Diamantis}
\address{ International College Beijing,
China Agricultural University,
No.17 Qinghua East Road, Haidian District,
Beijing, {100083}, P. R. China.}
\email{ioannis.diamantis@hotmail.com}



\setcounter{section}{-1}

\date{}

\begin{abstract}
In this paper we work toward the HOMFLYPT skein module of $L(p, 1)$, $\mathcal{S}(L(p,1))$, via braids. Our starting point is the linear Turaev-basis, $\Lambda^{\prime}$, of the HOMFLYPT skein module of the solid torus ST, $\mathcal{S}({\rm ST})$, which can be decomposed as the tensor product of the ``positive'' ${\Lambda^{\prime}}^+$ and the ``negative'' ${\Lambda^{\prime}}^-$ sub-modules, and the Lambropoulou invariant, $X$, for knots and links in ST, that captures $S({\rm ST})$. It is a well-known result by now that $\mathcal{S}(L(p, 1))=\frac{\mathcal{S}(ST)}{<bbm's>}$, where bbm's (braid band moves) denotes the isotopy moves that correspond to the surgery description of $L(p, 1)$. Namely, a HOMFLYPT-type invariant for knots and links in ST can be extended to an invariant for knots and links in $L(p, 1)$ by imposing relations coming from the performance of bbm's and solving the infinite system of equations obtained that way.
\smallbreak
In this paper we work with a new basis of $\mathcal{S}({\rm ST})$, $\Lambda$, and we relate the infinite system of equations obtained by performing bbm's on elements in $\Lambda^+$ to the infinite system of equations obtained by performing bbm's on elements in $\Lambda^-$ via a map $I$. More precisely we prove that the solutions of one system can be derived from the solutions of the other. Our aim is to reduce the complexity of the infinite system one needs to solve in order to compute $\mathcal{S}(L(p,1))$ using the braid technique. Finally, we present a generating set and a potential basis for $\frac{\Lambda^+}{<bbm's>}$ and thus, we obtain a generating set and a potential basis for $\frac{\Lambda^-}{<bbm's>}$. We also discuss further steps needed in order to compute $\mathcal{S}(L(p,1))$ via braids.
\end{abstract}

\maketitle

\section{Introduction and overview}\label{intro}

Skein modules were independently introduced by Przytycki \cite{P} and Turaev \cite{Tu} as generalizations of knot polynomials in $S^3$ to knot polynomials in arbitrary 3-manifolds. The essence is that skein modules are formal linear sums of (oriented) links in a 3-manifold $M$, modulo some local skein relations.

\begin{defn}\rm
Let $M$ be an oriented $3$-manifold, $R=\mathbb{Z}[u^{\pm1},z^{\pm1}]$, $\mathcal{L}$ the set of all oriented links in $M$ up to ambient isotopy in $M$ and let $S$ be the submodule of $R\mathcal{L}$ generated by the skein expressions $u^{-1}L_{+}-uL_{-}-zL_{0}$, where
$L_{+}$, $L_{-}$ and $L_{0}$ comprise a Conway triple represented schematically by the illustrations in Figure~\ref{skein}.

\begin{figure}[!ht]
\begin{center}
\includegraphics[width=1.7in]{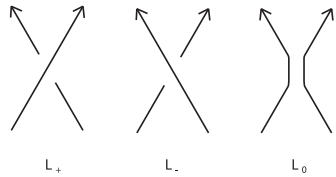}
\end{center}
\caption{The links $L_{+}, L_{-}, L_{0}$ locally.}
\label{skein}
\end{figure}

\noindent For convenience we allow the empty knot, $\emptyset$, and add the relation $u^{-1} \emptyset -u\emptyset =zT_{1}$, where $T_{1}$ denotes the trivial knot. Then the {\it HOMFLYPT skein module} of $M$ is defined to be:

\begin{equation*}
\mathcal{S} \left(M\right)=\mathcal{S} \left(M;{\mathbb Z}\left[u^{\pm 1} ,z^{\pm 1} \right],u^{-1} L_{+} -uL_{-} -zL{}_{0} \right)={\raise0.7ex\hbox{$
R\mathcal{L} $}\!\mathord{\left/ {\vphantom {R\mathcal{L} S }} \right. \kern-\nulldelimiterspace}\!\lower0.7ex\hbox{$ S  $}}.
\end{equation*}

\end{defn}

\noindent The HOMFLYPT skein module of a 3-manifold is very hard to compute (see \cite{HP} for a survey on skein modules). For example, $\mathcal{S}(S^3)$ is freely generated by the unknot (\cite{FYHLMO, PT}). Let now ST denote the solid torus. In \cite{Tu}, \cite{HK} the Homflypt skein module of the solid torus has been computed using diagrammatic methods by means of the following theorem:

\begin{thm}[Turaev, Kidwell--Hoste] \label{turaev}
The skein module $\mathcal{S}({\rm ST})$ is a free, infinitely generated $\mathbb{Z}[u^{\pm1},z^{\pm1}]$-module isomorphic to the symmetric
tensor algebra $SR\widehat{\pi}^0$, where $\widehat{\pi}^0$ denotes the conjugacy classes of non trivial elements of $\pi_1(\rm ST)$.
\end{thm}

A basic element of $\mathcal{S}({\rm ST})$ in the context of \cite{Tu, HK}, is illustrated in Figure~\ref{tur}. Note that in the diagrammatic setting of \cite{Tu} and \cite{HK}, ST is considered as ${\rm Annulus} \times {\rm Interval}$. 

\begin{figure}
\begin{center}
\includegraphics[width=1.3in]{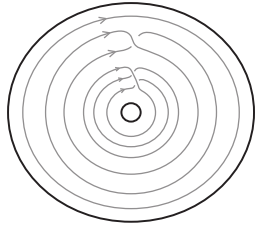}
\end{center}
\caption{A basic element of $\mathcal{S}({\rm ST})$.}
\label{tur}
\end{figure}

\smallbreak

$\mathcal{S}({\rm ST})$ is well-studied and understood by now. It forms a commutative algebra with multiplication induced by embedding two solid tori in one in a standard way. Let now $\mathcal{B}^+$ denote the sub-algebra of $\mathcal{S}({\rm ST})$, freely generated by elements that are clockwise oriented (see Fig.~\ref{tur}) and let $\mathcal{B}^-$ denote the sub-algebra of $\mathcal{S}({\rm ST})$, freely generated by elements with counter-clockwise orientation. Let also $\mathcal{B}_k^+$ denote the sub-module generated by elements in $\mathcal{B}^+$ whose winding number is equal to $k\in \mathbb{N}$ and $\mathcal{B}_{-k}^-$ denote the sub-module generated by elements in $\mathcal{B}^-$ whose winding number is equal to $k$. As a linear space, $\mathcal{B}^+$ is graded by $$\mathcal{B}^+ \cong \underset{k\geq 0}{\oplus}\, \mathcal{B}_k^+$$

\noindent and similarly, $\mathcal{B}^-$ is graded by $$\mathcal{B}^- \cong \underset{k\geq 0}{\oplus}\, \mathcal{B}_{-k}^-.$$
Finally, we have the following module decomposition:

\[
\mathcal{S}({\rm ST})\ =\ \underset{\lambda, \mu\geq 0}{\oplus}\, \mathcal{B}_{-\lambda}^-\, \otimes\, \mathcal{B}_{\mu}^+.
\]

The Turaev-basis of $\mathcal{S}({\rm ST})$ is described in Equation~\ref{Lpr} in open braid form (see left illustration of Figure~\ref{basel}). In \cite{DL2} a new basis, $\Lambda$, of $\mathcal{S}({\rm ST})$ is presented via braids, that naturally describes isotopy in $L(p, 1)$. This basis was obtained by relating the Turaev-basis to $\Lambda$ via a lower triangular matrix with invertible elements in the diagonal, and is presented in Equation~\ref{basis} in open braid form (see right illustration of Figure~\ref{basel}). The sets $\mathcal{B}^+$ and $\mathcal{B}^-$ are presented in Equation~\ref{Bpn} and the sets $\mathcal{B}_{k}^+$ and $\mathcal{B}_{-k}^-$ are presented in Equations~\ref{lamaugp} and \ref{lamaugn} respectively.

\smallbreak

In \cite{DLP} the relation between $\mathcal{S}({\rm ST})$ and $\mathcal{S}(L(p, 1))$ is established and it is shown that:

\begin{equation}\label{inft}
\mathcal{S}(L(p, 1))=\frac{\mathcal{S}({\rm ST})}{<bbm_i>}\, ,
\end{equation}

\noindent where $<bbm_i>$ corresponds to the relations coming from the performance of all possible braid band moves (or slide moves) on elements in a basis of $\mathcal{S}({\rm ST})$. More precisely, Eq.~(\ref{inft}) suggests that in order to compute $\mathcal{S}(L(p, 1))$, we need to consider elements in $\mathcal{S}({\rm ST})$, apply all possible bbm's and identify all linear dependent elements. A step toward a simplification of the above infinite system of equations can be found in \cite{DL4}, where it is shown that in order to compute $\mathcal{S}(L(p, 1))$ it suffices to consider elements in an augmented set, $\Lambda^{aug}$ (Equation~\ref{lamaug},) and perform bbm's only on their first moving strand (the strand that lies closer to the surgery strand), i.e. $\mathcal{S}(L(p, 1))=\frac{\Lambda^{aug}}{<bbm_1>}$. In that way more control over the infinite system is obtained.

\smallbreak

In this paper we consider the module obtained by solving the infinite system of equations (\ref{inft}), where we only consider elements in $\Lambda^{aug_+}$, a set related to $\mathcal{B}^+$ and that is presented in Eq.~\ref{lamaugp}, and we perform braid band moves on their first moving strands, namely $\frac{\Lambda^{aug_+}}{<bbm_1>}$. Similarly, we consider the module $\frac{\Lambda^{aug_-}}{<bbm_1>}$ obtained by solving the infinite system of equations by considering elements in $\Lambda^{aug_-}$ (Definition~\ref{ponel}, Eq.~\ref{lamaugn}), and we perform braid band moves on their first moving strands. We then relate these modules via two maps $f$ and $I$ defined in Definition~\ref{mapfI}, and in particular we show that the solution of the system $\frac{\Lambda^{aug_-}}{<bbm_1>}$ can be derived from the solution of the system $\frac{\Lambda^{aug_+}}{<bbm_1>}$. In this way we simplify the infinite system (\ref{inft}). Furthermore, we provide potential bases for $\frac{\Lambda^{aug_+}}{<bbm_1>}$ and $\frac{\Lambda^{aug_-}}{<bbm_1>}$ and we present results suggesting that the full solution of the infinite system (\ref{inft}) would correspond to the basis of $\mathcal{S}(L(p, 1))$ presented in \cite{GM}, and which was obtained using diagrammatic methods. Finally, the braid technique can also be applied for computing other types of skein modules. The interested reader is referred to \cite{D2} and \cite{D3} for the case of Kauffman bracket skein modules of 3-manifolds. 

\smallbreak

The paper is organized as follows: In \S\ref{prel} we discuss isotopy \& braid equivalence for knots and links in $L(p, 1)$ (\cite{DL1, LR1}) and in \S\ref{SolidTorus} we recall the setting and the essential techniques and results from \cite{La1, La2, LR1, LR2, DL1, DL2, DL3, DL4, DLP} in order to describe the HOMFLYPT skein module of ST via braids. In particular, we present the generalized Hecke algebra of type B, and through a unique trace defined on this algebra, we present the Lambropoulou invariant for knots and links in ST that captures $\mathcal{S}({\rm ST})$. In \S\ref{ordsec} we present an ordering relation defined on $\mathcal{S}({\rm ST})$, which is crucial in order to obtain the new basis of $\mathcal{S}({\rm ST})$, $\Lambda$, and in \S\ref{turtolam} we present results from \cite{DL2} that are used in order to relate the sets $\Lambda$ and $\Lambda^{\prime}$ via an infinite triangular matrix with invertible elements in the diagonal. Moreover, in \S\ref{homlen} we briefly discuss the relation of $\mathcal{S}(L(p, 1))$ to $\mathcal{S}({\rm ST})$ presented in \cite{DLP, DL4}. In \S\ref{relposneg} we relate the modules $\frac{\Lambda^{aug}_{+}}{<bbm_1>}$ and $\frac{\Lambda^{aug}_{-}}{<bbm_1>}$ via the maps $f$ and $I$ (see \S\ref{msec}) and in \S\ref{slp1} we present generating sets and the potential bases of these modules. Finally, in \S\ref{further} we present further results toward the solution of the infinite system (that is studied in \cite{DL5}), that lead to the \cite{GM}-basis of $\mathcal{S}(L(p, 1))$.

\section{Preliminaries}\label{prel}

\subsection{Mixed Links in $S^3$}

We consider ST to be the complement of a solid torus in $S^3$ and knots in ST are represented by \textit{mixed links} in $S^{3}$. Mixed links consist of two parts, the unknotted fixed part $\widehat{I}$ that represents the complementary solid torus in $S^3$ and the moving part $L$ that links with $\widehat{I}$. A \textit{mixed link diagram} is a diagram $\widehat{I}\cup \widetilde{L}$ of $\widehat{I}\cup L$ on the plane of $\widehat{I}$, where this plane is equipped with the top-to-bottom direction of $I$ (see top left hand side of Figure~\ref{bmov}). For more details on mixed links the reader is referred to \cite{LR1, LR2, DL1} and references therein.

\smallbreak

The lens spaces $L(p, 1)$ can be obtained from $S^3$ by surgery on the unknot with integer surgery coefficient $p$. Surgery along the unknot can be realized by considering first the complementary solid torus and then attaching to it a solid torus according to some homeomorphism on the boundary. Thus, isotopy in $L(p, 1)$ can be viewed as isotopy in ST together with the band moves in $S^3$, which reflect the surgery description of $L(p, 1)$. Moreover, in \cite{DL1} it is shown that in order to describe isotopy for knots and links in a c.c.o. $3$-manifold, it suffices to consider only the type $\alpha$ band moves (for an illustration see top of Figure~\ref{bmov}) and thus, isotopy between oriented links in $L(p, 1)$ is reflected in $S^3$ by means of the following result (cf. Thm.~5.8 \cite{LR1}, Thm.~6 \cite{DL1} ):

\smallbreak

{\it
Two oriented links in $L(p, 1)$ are isotopic if and only if two corresponding mixed link diagrams of theirs differ by isotopy in {\rm ST} together with a finite sequence of the type $\alpha$ band moves.
}

\subsection{Mixed braids and braid equivalence for knots and links in $L(p,1)$}

By the Alexander theorem for knots and links in the solid torus (cf. Thm.~1 \cite{La2}), a mixed link diagram $\widehat{I}\cup \widetilde{L}$ of $\widehat{I}\cup L$ may be turned into a \textit{mixed braid} $I\cup \beta$ with isotopic closure. This is a braid in $S^{3}$ where, without loss of generality, its first strand represents $\widehat{I}$, the fixed part, and the other strands, $\beta$, represent the moving part $L$. The subbraid $\beta$ is called the \textit{moving part} of $I\cup \beta$ (see bottom left hand side of Figure~\ref{bmov}). Then, in order to translate isotopy for links in $L(p,1)$ into braid equivalence, we first perform the technique of {\it standard parting} introduced in \cite{LR2} in order to separate the moving strands from the fixed strand that represents the lens spaces $L(p,1)$. This can be realized by pulling each pair of corresponding moving strands to the right and {\it over\/} or {\it under\/} the fixed strand that lies on their right. Then, we define a {\it braid band move} to be a move between mixed braids, which is a band move between their closures. It starts with a little band oriented downward, which, before sliding along a surgery strand, gets one twist {\it positive\/} or {\it negative\/} (see bottom of Figure~\ref{bmov}).

\begin{figure}
\begin{center}
\includegraphics[width=3.4in]{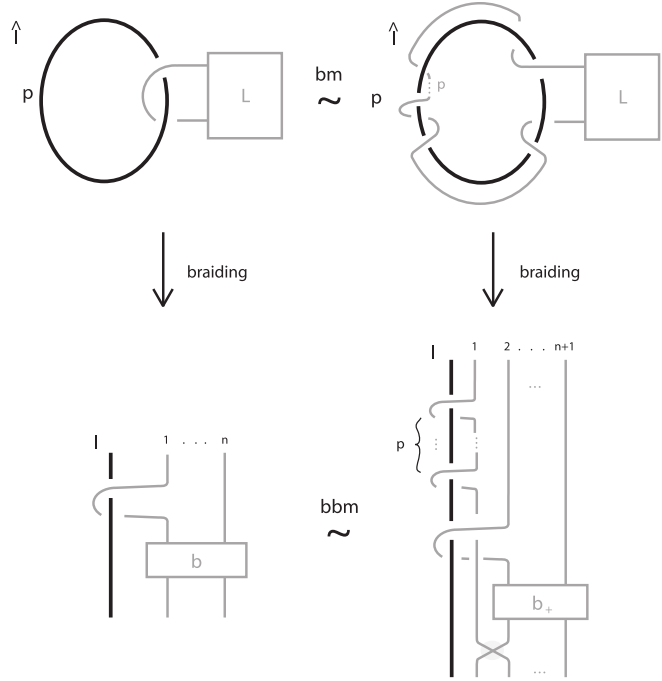}
\end{center}
\caption{Isotopy in $L(p,1)$ and the two types of braid band moves on mixed braids.}
\label{bmov}
\end{figure}

\smallbreak

The sets of braids related to ST form groups, which are in fact the Artin braid groups of type B, denoted $B_{1,n}$, with presentation:

\[ B_{1,n} = \left< \begin{array}{ll}  \begin{array}{l} t, \sigma_{1}, \ldots ,\sigma_{n-1}  \\ \end{array} & \left| \begin{array}{l}
\sigma_{1}t\sigma_{1}t=t\sigma_{1}t\sigma_{1} \ \   \\
 t\sigma_{i}=\sigma_{i}t, \quad{i>1}  \\
{\sigma_i}\sigma_{i+1}{\sigma_i}=\sigma_{i+1}{\sigma_i}\sigma_{i+1}, \quad{ 1 \leq i \leq n-2}   \\
 {\sigma_i}{\sigma_j}={\sigma_j}{\sigma_i}, \quad{|i-j|>1}  \\
\end{array} \right.  \end{array} \right>, \]

\noindent where the generators $\sigma _{i}$ and $t$ are illustrated in Figure~\ref{genh}(i).

Let now $\mathcal{L}$ denote the set of oriented knots and links in ST. Then, isotopy in $L(p,1)$ is then translated on the level of mixed braids by means of the following theorem:

\begin{thm}[Theorem~5, \cite{LR2}] \label{markov}
 Let $L_{1} ,L_{2}$ be two oriented links in $L(p,1)$ and let $I\cup \beta_{1} ,{\rm \; }I\cup \beta_{2}$ be two corresponding mixed braids in $S^{3}$. Then $L_{1}$ is isotopic to $L_{2}$ in $L(p,1)$ if and only if $I\cup \beta_{1}$ is equivalent to $I\cup \beta_{2}$ in $\mathcal{B}$ by the following moves:
\[ \begin{array}{clll}
(i)  & Conjugation:         & \alpha \sim \beta^{-1} \alpha \beta, & {\rm if}\ \alpha ,\beta \in B_{1,n}. \\
(ii) & Stabilization\ moves: &  \alpha \sim \alpha \sigma_{n}^{\pm 1} \in B_{1,n+1}, & {\rm if}\ \alpha \in B_{1,n}. \\
(iii) & Loop\ conjugation: & \alpha \sim t^{\pm 1} \alpha t^{\mp 1}, & {\rm if}\ \alpha \in B_{1,n}. \\
(iv) & Braid\ band\ moves: & \alpha \sim {t}^p \alpha_+ \sigma_1^{\pm 1}, & a_+\in B_{1, n+1},
\end{array} \]

\noindent where $\alpha_+$ is the word $\alpha$ with all indices shifted by +1. Note that moves (i), (ii) and (iii) correspond to link isotopy in {\rm ST}.
\end{thm}

\begin{nt}\rm
We denote a braid band move by bbm and, specifically, the result of a positive or negative braid band move performed on the $i^{th}$-moving strand of a mixed braid $\beta$ by $bbm_{\pm i}(\beta)$.
\end{nt}

Note also that in \cite{LR2} it was shown that the choice of the position of connecting the two components after the performance of a bbm is arbitrary.

\section{The HOMFLYPT skein module of ST via braids}\label{SolidTorus}

In \cite{La2} the most generic analogue of the HOMFLYPT polynomial, $X$, for links in the solid torus $\rm ST$ has been derived from the generalized Iwahori--Hecke algebras of type $\rm B$, $\textrm{H}_{1,n}$, via a unique Markov trace constructed on them. This algebra was defined as the quotient of ${\mathbb C}\left[q^{\pm 1} \right]B_{1,n}$ over the quadratic relations ${g_{i}^2=(q-1)g_{i}+q}$. Namely:

\begin{equation*}
\textrm{H}_{1,n}(q)= \frac{{\mathbb C}\left[q^{\pm 1} \right]B_{1,n}}{ \langle \sigma_i^2 -\left(q-1\right)\sigma_i-q \rangle}.
\end{equation*}

It is also shown that the following sets form linear bases for ${\rm H}_{1,n}(q)$ (\cite[Proposition~1 \& Theorem~1]{La2}):

\[
\begin{array}{llll}
 (i) & \Sigma_{n} & = & \{t_{i_{1} } ^{k_{1} } \ldots t_{i_{r}}^{k_{r} } \cdot \sigma \} ,\ {\rm where}\ 0\le i_{1} <\ldots <i_{r} \le n-1,\\
 (ii) & \Sigma^{\prime} _{n} & = & \{ {t^{\prime}_{i_1}}^{k_{1}} \ldots {t^{\prime}_{i_r}}^{k_{r}} \cdot \sigma \} ,\ {\rm where}\ 0\le i_{1} < \ldots <i_{r} \le n-1, \\
\end{array}
\]
\noindent where $k_{1}, \ldots ,k_{r} \in {\mathbb Z}$, $t_0^{\prime}\ =\ t_0\ :=\ t, \quad t_i^{\prime}\ =\ g_i\ldots g_1tg_1^{-1}\ldots g_i^{-1} \quad {\rm and}\quad t_i\ =\ g_i\ldots g_1tg_1\ldots g_i$ are the `looping elements' in ${\rm H}_{1, n}(q)$ (see Figure~\ref{genh}(ii)) and $\sigma$ a basic element in the Iwahori--Hecke algebra of type A, ${\rm H}_{n}(q)$, for example in the form of the elements in the set \cite{Jo}:

$$ S_n =\left\{(g_{i_{1} }g_{i_{1}-1}\ldots g_{i_{1}-k_{1}})(g_{i_{2} }g_{i_{2}-1 }\ldots g_{i_{2}-k_{2}})\ldots (g_{i_{p} }g_{i_{p}-1 }\ldots g_{i_{p}-k_{p}})\right\}, $$

\noindent for $1\le i_{1}<\ldots <i_{p} \le n-1{\rm \; }$. In \cite{La2} the bases $\Sigma^{\prime}_{n}$ are used for constructing a Markov trace on $\mathcal{H}:=\bigcup_{n=1}^{\infty}{\rm H}_{1, n}$, and using this trace, a universal HOMFLYPT-type invariant for oriented links in ST was constructed.

\begin{figure}
\begin{center}
\includegraphics[width=5.1in]{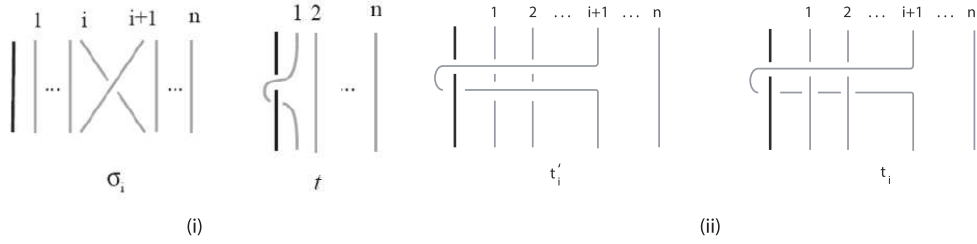}
\end{center}
\caption{The generators of $B_{1, n}$ and the `looping' elements $t^{\prime}_{i}$ and $t_{i}$.}
\label{genh}
\end{figure}

\begin{thm}{\cite[Theorem~6 \& Definition~1]{La2}} \label{tr}
Given $z, s_{k}$ with $k\in {\mathbb Z}$ specified elements in $R={\mathbb C}\left[q^{\pm 1} \right]$, there exists a unique linear Markov trace function on $\mathcal{H}$:

\begin{equation*}
{\rm tr}:\mathcal{H}  \to R\left(z,s_{k} \right),\ k\in {\mathbb Z}
\end{equation*}

\noindent determined by the rules:

\[
\begin{array}{lllll}
(1) & {\rm tr}(ab) & = & {\rm tr}(ba) & \quad {\rm for}\ a,b \in {\rm H}_{1,n}(q) \\
(2) & {\rm tr}(1) & = & 1 & \quad {\rm for\ all}\ {\rm H}_{1,n}(q) \\
(3) & {\rm tr}(ag_{n}) & = & z{\rm tr}(a) & \quad {\rm for}\ a \in {\rm H}_{1,n}(q) \\
(4) & {\rm tr}(a{t^{\prime}_{n}}^{k}) & = & s_{k}{\rm tr}(a) & \quad {\rm for}\ a \in {\rm H}_{1,n}(q),\ k \in {\mathbb Z} \\
\end{array}
\]

\bigbreak

\noindent Then, the function $X:\mathcal{L}$ $\rightarrow R(z,s_{k})$

\begin{equation*}
X_{\widehat{\alpha}} = \Delta^{n-1}\cdot \left(\sqrt{\lambda } \right)^{e}
{\rm tr}\left(\pi \left(\alpha \right)\right),
\end{equation*}

\noindent is an invariant of oriented links in {\rm ST}, where $\Delta:=-\frac{1-\lambda q}{\sqrt{\lambda } \left(1-q\right)}$, $\lambda := \frac{z+1-q}{qz}$, $\alpha \in B_{1,n}$ is a word in the $\sigma _{i}$'s and $t^{\prime}_{i} $'s, $\widehat{\alpha}$ is the closure of $\alpha$, $e$ is the exponent sum of the $\sigma _{i}$'s in $\alpha $, $\pi$ the canonical map of $B_{1,n}$ on ${\rm H}_{1,n}(q)$, such that $t\mapsto t$ and $\sigma _{i} \mapsto g_{i}$.
\end{thm}

\begin{remark}\rm
Note that the use of the looping elements $t^{\prime}$'s enable the trace to be defined by just extending by rule (4) the three rules of the Ocneanu trace on the algebras ${\rm H}_n(q)$ (\cite{Jo}).
\end{remark}

\bigbreak

In the braid setting of \cite{La2}, the elements of $\mathcal{S}({\rm ST})$ correspond bijectively to the elements of the following set $\Lambda^{\prime}$:

\begin{equation}\label{Lpr}
\Lambda^{\prime}=\{ {t^{k_0}}{t^{\prime}_1}^{k_1} \ldots
{t^{\prime}_n}^{k_n}, \ k_i \in \mathbb{Z}\setminus\{0\}, \ k_i \leq k_{i+1}\ \forall i,\ n\in \mathbb{N} \}.
\end{equation}

\noindent As explained in \cite{La2, DL2}, the set $\Lambda^{\prime}$ forms a basis of $\mathcal{S}({\rm ST})$ in terms of braids (see also \cite{HK, Tu}). Note that $\Lambda^{\prime}$ is a subset of $\mathcal{H}$ and, in particular, $\Lambda^{\prime}$ is a subset of $\Sigma^{\prime}=\bigcup_n\Sigma^{\prime}_n$. Note also that in contrast to elements in $\Sigma^{\prime}$, the elements in $\Lambda^{\prime}$ have no gaps in the indices, the exponents are ordered and there are no `braiding tails'. 

\begin{remark}\rm
The Lambropoulou invariant $X$ recovers $\mathcal{S}({\rm ST})$. Indeed, it gives distinct values to distinct elements of $\Lambda^{\prime}$, since ${\rm tr}(t^{k_0}{t^{\prime}_1}^{k_1} \ldots {t^{\prime}_n}^{k_n})=s_{k_n}\ldots s_{k_1}s_{k_0}$.
\end{remark}

\subsection{A different basis for $\mathcal{S}({\rm ST})$}

In \cite{DL2}, a different basis $\Lambda$ for $\mathcal{S}({\rm ST})$ is presented, which is crucial toward the computation of $\mathcal{S}\left(L(p,1)\right)$ and which is described in Eq.~\ref{basis} in open braid form (for an illustration see Figure~\ref{basel}). In particular we have the following:

\begin{thm}{\cite[Theorem~2]{DL2}}\label{newbasis}
The following set is a $\mathbb{C}[q^{\pm1}, z^{\pm1}]$-basis for $\mathcal{S}({\rm ST})$:
\begin{equation}\label{basis}
\Lambda=\{t^{k_0}t_1^{k_1}\ldots t_n^{k_n},\ k_i \in \mathbb{Z}\setminus\{0\},\ k_i \leq k_{i+1}\ \forall i,\ n \in \mathbb{N} \}.
\end{equation}
\end{thm}

The importance of the new basis $\Lambda$ of $\mathcal{S}({\rm ST})$ lies in the simplicity of the algebraic expression of a braid band move, which extends the link isotopy in ST to link isotopy in $L(p,1)$ and this fact was our motivation for establishing this new basis $\Lambda$. Note that comparing the set $\Lambda$ with the set $\Sigma=\bigcup_n\Sigma_n$, we observe that in $\Lambda$ there are no gaps in the indices of the $t_i$'s and the exponents are in decreasing order. Also, there are no `braiding tails' in the words in $\Lambda$.

\begin{figure}
\begin{center}
\includegraphics[width=1.6in]{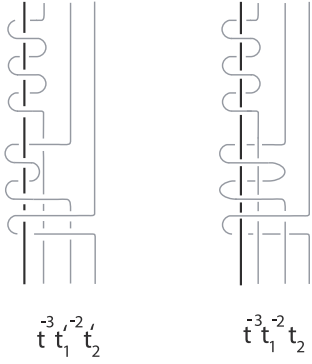}
\end{center}
\caption{Elements in the two different bases of $\mathcal{S}({\rm ST})$.}
\label{basel}
\end{figure}

\subsection{An ordering in the bases of $\mathcal{S}({\rm ST})$}\label{ordsec}

We now define an ordering relation in the sets $\Sigma$ and $\Sigma^{\prime}$, which passes to their respective subsets $\Lambda$ and $\Lambda^{\prime}$ and that first appeared in \cite{DL2}. This ordering relation plays a crucial role to what will follow. For that we need the notion of the {\it index} of a word $w$ in any of these sets, denoted $ind(w)$. In $\Lambda^{\prime}$ or $\Lambda$ $ind(w)$ is defined to be the highest index of the $t_i^{\prime}$'s, resp. of the $t_i$'s in $w$. Similarly, in $\Sigma^{\prime}$ or $\Sigma$, $ind(w)$ is defined as above by ignoring possible gaps in the indices of the looping generators and by ignoring the braiding parts in the algebras $\textrm{H}_{n}(q)$. Moreover, the index of a monomial in $\textrm{H}_{n}(q)$ is equal to $0$.

\begin{defn}{\cite[Definition~2]{DL2}} \label{order}
\rm
Let $w={t^{\prime}_{i_1}}^{k_1}\ldots {t^{\prime}_{i_{\mu}}}^{k_{\mu}}\cdot \beta_1$ and $u={t^{\prime}_{j_1}}^{\lambda_1}\ldots {t^{\prime}_{j_{\nu}}}^{\lambda_{\nu}}\cdot \beta_2$ in $\Sigma^{\prime}$, where $k_t , \lambda_s \in \mathbb{Z}$ for all $t,s$ and $\beta_1, \beta_2 \in H_n(q)$. Then, we define the following ordering in $\Sigma^{\prime}$:

\smallbreak

\begin{itemize}
\item[(a)] If $\sum_{i=0}^{\mu}k_i < \sum_{i=0}^{\nu}\lambda_i$, then $w<u$.

\vspace{.1in}

\item[(b)] If $\sum_{i=0}^{\mu}k_i = \sum_{i=0}^{\nu}\lambda_i$, then:

\vspace{.1in}

\noindent  (i) if $ind(w)<ind(u)$, then $w<u$,

\vspace{.1in}

\noindent  (ii) if $ind(w)=ind(u)$, then:

\vspace{.1in}

\noindent \ \ \ \ ($\alpha$) if $i_1=j_1, \ldots , i_{s-1}=j_{s-1}, i_{s}<j_{s}$, then $w>u$,

\vspace{.1in}

\noindent \ \ \  ($\beta$) if $i_t=j_t$ for all $t$ and $k_{\mu}=\lambda_{\mu}, k_{\mu-1}=\lambda_{\mu-1}, \ldots, k_{i+1}=\lambda_{i+1}, |k_i|<|\lambda_i|$, then $w<u$,

\vspace{.1in}

\noindent \ \ \  ($\gamma$) if $i_t=j_t$ for all $t$ and $k_{\mu}=\lambda_{\mu}, k_{\mu-1}=\lambda_{\mu-1}, \ldots, k_{i+1}=\lambda_{i+1}, |k_i|=|\lambda_i|$ and $k_i>\lambda_i$, then $w<u$,

\vspace{.1in}

\noindent \ \ \ \ ($\delta$) if $i_t=j_t\ \forall t$ and $k_i=\lambda_i$, $\forall i$, then $w=u$.

\end{itemize}

The ordering in the set $\Sigma$ is defined as in $\Sigma^{\prime}$, where $t_i^{\prime}$'s are replaced by $t_i$'s.
\end{defn}

\begin{nt}\label{nt} \rm
We set $\tau_{i,i+m}^{k_{i,i+m}}:=t_i^{k_i}\ldots t^{k_{i+m}}_{i+m}$ and ${\tau^{\prime}}_{i,i+m}^{k_{i,i+m}}:={t^{\prime}}_i^{k_i}\ldots {t^{\prime}}^{k_{i+m}}_{i+m}$, for $m\in \mathbb{N}$, $k_j\neq 0$ for all $j$.
\end{nt}

\bigbreak

The \textit{subsets of level $k$}, $\Lambda_{(k)}$ and $\Lambda^{\prime}_{(k)}$, of $\Lambda$ and $\Lambda^{\prime}$ respectively ({\cite[Definition~3]{DL2}}), are defined to be the sets:

\begin{equation}\label{levk}
\begin{array}{l}
\Lambda_{(k)}:=\{t_0^{k_0}t_1^{k_1}\ldots t_{m}^{k_m} | \sum_{i=0}^{m}{k_i}=k,\ k_i \in \mathbb{Z}\setminus\{0\},\  k_i \leq k_{i+1}\ \forall i \}\\
\\
\Lambda^{\prime}_{(k)}:=\{{t^{\prime}_0}^{k_0}{t^{\prime}_1}^{k_1}\ldots {t^{\prime}_{m}}^{k_m} | \sum_{i=0}^{m}{k_i}=k,\ k_i \in \mathbb{Z}\setminus\{0\},\  k_i \leq k_{i+1}\ \forall i \}
\end{array}
\end{equation}

\noindent In \cite{DL2} it was shown that the sets $\Lambda_{(k)}$ and $\Lambda^{\prime}_{(k)}$ are totally ordered and well ordered for all $k$ (\cite[Propositions~1 \& 2]{DL2}). Note that in \cite{DLP} the exponents in the monomials of $\Lambda$ are in decreasing order, while here the exponents are considered in increasing order, which is totally symmetric. 

\bigbreak

We finally define the set $\Lambda^{aug}$, which augments the basis $\Lambda$ and its subset of level $k$, and we also introduce the notion of \textit{homologous words}.

\begin{defn}\rm
We define the set:
\begin{equation}\label{lamaug}
\Lambda^{aug}\ :=\{t_0^{k_0}t_1^{k_1}\ldots t_{n}^{k_n},\ k_i \in \mathbb{Z}\backslash \{0\}\}.
\end{equation}
\noindent and the {\it subset of level} $k$, $\Lambda^{aug}_{(k)}$, of $\Lambda^{aug}$:

\begin{equation}\label{lamaug2}
\Lambda^{aug}_{(k)}:=\{t_0^{k_0}t_1^{k_1}\ldots t_{m}^{k_m} | \sum_{i=0}^{m}{k_i}=k,\ k_i \in \mathbb{Z}\backslash \{0\}\}
\end{equation}
\end{defn}

\begin{defn}\label{homw} \rm
We shall say that two words $w^{\prime}\in \Lambda^{\prime}$ and $w\in \Lambda$ are {\it homologous}, denoted $w^{\prime}\sim w$, if $w$ is
obtained from $w^{\prime}$ by turning $t^{\prime}_i$ into $t_i$ for all $i$.
\end{defn}

\subsection{Relating $\Lambda^{\prime}$ to $\Lambda$}\label{turtolam}

We now present results from \cite{DL2} used in order to relate the sets $\Lambda^{\prime}$ and $\Lambda$ via a lower triangular matrix with invertible elements in the diagonal. We start by expressing elements in $\Lambda^{\prime}$ to to expressions containing the $t_i$'s. We have that:

\begin{thm}[Theorem 7, \cite{DL2}]\label{convert}
The following relations hold in ${\rm H}_{1,n}(q)$ for $k \in \mathbb{Z}$:
\[
\begin{array}{lcl}
t^{k_0}{t_1^{\prime}}^{k_1} \ldots {t_m^{\prime}}^{k_m} & = & q^{-\underset{n=1}{\overset{m}{\sum}}\, {nk_n}}\cdot\ t^{k_0}t_1^{k_1}\ldots t_m^{k_m} \ + \ \underset{i}{\sum}\, {f_i(q)\cdot t^{k_0}t_1^{k_1}\ldots t_m^{k_m}\cdot w_i} \ +\\
&&\\
& + & \underset{j}{\sum}\, {g_j(q)\tau_j \cdot u_j},
\end{array}
\]
\noindent where $w_i, u_j \in {\rm H}_{m+1}(q), \forall i$, $\tau_j \in \Sigma_n$, such that $\tau_j < t^{k_0}t_1^{k_1}\ldots t_m^{k_m}, \forall j$.
\end{thm}

Equivalently, the relation in Theorem~\ref{convert} can be written as:

\begin{equation}\label{convert2}
\begin{array}{lcl}
 t^{k_0}t_1^{k_1}\ldots t_m^{k_m} & = & q^{\underset{n=1}{\overset{m}{\sum}}\, {nk_n}}\cdot\ t^{k_0}{t_1^{\prime}}^{k_1} \ldots {t_m^{\prime}}^{k_m}
\ + \ \underset{i}{\sum}\, {f^{\prime}_i(q)\cdot t^{k_0}t_1^{k_1}\ldots t_m^{k_m}\cdot w_i} \ +\\
&&\\
& + & \underset{j}{\sum}\, {g^{\prime}_j(q)\tau_j \cdot u_j},
\end{array}
\end{equation}

When applying Theorem~\ref{convert} on an element in $\Lambda^{\prime}$, we obtain the homologous word, the homologous word followed by a braiding ``tail'', and a sum of lower order terms followed by braiding ``tails''. These elements belong to $\Sigma_n$ since they may have gaps in their indices, and we manage the gaps applying Theorem~8 in \cite{DL2}, namely:

\begin{equation}\label{gaps}
\Sigma_n \ni \tau\ \widehat{=}\ \underset{i}{\sum}\, f_{i}(q)\, \tau_i\cdot w_i\ :\ \tau_i\in \Lambda^{aug}\,\ w_i\in {\rm H}_{n}(q),\ \forall i,
\end{equation}

\noindent where $\widehat{=}$ denotes that conjugation is applied in this process.

\smallbreak

We now deal with the elements in $\Lambda^{aug}_{(k)}$ that are followed by a braiding ``tail'' $w$ in $H_n(q)$. More precisely we have:

\begin{thm}[Theorem~9, \cite{DL2}]\label{tails}
For an element in $\Lambda^{aug}_{(k)}$ followed by a braiding ``tail'' $w$ in $H_n(q)$ we have that:

$$tr(\tau \cdot w) \ =\  \sum_{j}{f_j(q,z)\cdot tr(\tau_j)},$$

\noindent such that $\tau_j\in \Lambda^{aug}_{(k)}$ and $\tau_{j} < \tau$, for all $j$.
\end{thm}

One very important result in \cite{DL2} is that one can change the order of the exponents by using conjugation and stabilization moves on elements in $\Lambda$ and express them as sums of monomials in $t_i$'s with arbitrary exponents and which are of lower order than the initial elements in $\Lambda$. Note that both conjugation and stabilization moves are captures by the trace rules, and that we translate here Theorem~9 in \cite{DL2} using the trace. 

\begin{thm}{\cite[Theorem~9]{DL2}}\label{exp}
For an element in $\Lambda^{aug}$ followed by a braiding ``tail'' in ${\rm H}_n(q)$ we have that:

$$tr(\tau_{0,m}^{k_{0,m}}\cdot w)\ =\ tr\left(\underset{j}{\sum}\, {\tau_{0,j}^{\lambda_{0,j}}\cdot w_j}\right),$$

\noindent where $\tau_{0,j}^{\lambda_{0,j}} \in \Lambda$ and $w, w_j \in \bigcup_{n\in \mathbb{N}}{\rm H}_n(q)$ for all $j$.
\end{thm}

Examples of how to apply Theorems~\ref{convert}, \ref{tails} and \ref{exp} can be found in \cite{DL3}.

\subsection{Relating $\mathcal{S}\left(L(p, 1) \right)$ to $\mathcal{S}({\rm ST})$.}\label{homlen}

In order to simplify this system of equations (\ref{inft}), in \cite{DLP} we first show that performing a bbm on a mixed braid in $B_{1, n}$ reduces to performing bbm's on elements in the canonical basis, $\Sigma_n^{\prime}$, of the algebra ${\rm H}_{1,n}(q)$ and, in fact, on their first moving strand. We then reduce the equations obtained from elements in $\Sigma^{\prime}$ to equations obtained from elements in $\Sigma$. In order now to reduce further the computation to elements in the basis $\Lambda$ of $\mathcal{S}({\rm ST})$, in \cite{DLP} we manage the gaps in the indices of the looping generators of elements in $\Sigma$, obtaining elements in the augmented ${\rm H}_{n}(q)$-module $\Lambda^{aug}$, denoted by $\Lambda^{aug}|{\rm H}_n$. We need to emphasize on the fact that the ``managing the gaps'' procedure, allows the performance of bbm's to take place on {\it any moving} strand. Then, these equations are shown to be equivalent to equations obtained from elements in the ${\rm H}_{n}(q)$-module $\Lambda$, denoted by $\Lambda|{\rm H}_n$, by performing bbm's on any moving strand. We finally eliminate the braiding ``tails'' from elements in $\Lambda|{\rm H}_n$ and reduce the computations to the set $\Lambda$, where the bbm's are performed on any moving strand (see \cite{DLP}). Thus, in order to compute $\mathcal{S}(L(p,1))$, it suffices to solve the infinite system of equations obtained by performing bbm's on any moving strand of elements in the set $\Lambda$. Moreover, in \cite{DL4} we consider the augmented set $\Lambda^{aug}$ and show that the system of equations obtained from elements in $\Lambda$ by performing bbm's on {\it any moving} strand, is equivalent to the system of equations obtained by performing bbm's on the {\it first moving} strand of elements in $\Lambda^{aug}$. It is worth mentioning that although $\Lambda^{aug} \supset \Lambda$, the advantage of considering elements in the augmented set $\Lambda^{aug}$ is that we restrict the performance of the braid band moves only on the first moving strand and, thus, we obtain less equations and more control on the infinite system (\ref{inft}). 

\smallbreak

The above are summarized in the following sequence of equations:

$$
\begin{array}{llllll}
\mathcal{S}\left( L(p,1) \right) & = & \frac{\mathcal{S}({\rm ST})}{<a - bbm_i(a)>},\ a\in B_{1, n}, \ \forall\ i & = & \frac{\mathcal{S}({\rm ST})}{<s^{\prime} - bbm_1(s^{\prime})>},\ s^{\prime}\in \Sigma_n^{\prime} & = \\
&&&&&\\
& = & \frac{\mathcal{S}({\rm ST})}{<s - bbm_1(s)>},\ s\in \Sigma_n & = & \frac{\mathcal{S}({\rm ST})}{<\lambda^{\prime} - bbm_i(\lambda^{\prime})>},\ \lambda^{\prime} \in \Lambda^{aug}|{\rm H}_n,\ \forall\ i & = \\
&&&&&\\
& = & \frac{\mathcal{S}({\rm ST})}{<\lambda^{\prime \prime} - bbm_i(\lambda^{\prime \prime})>},\ \lambda^{\prime \prime}\in \Lambda|{\rm H}_n, \ \forall\ i & = & \frac{\mathcal{S}({\rm ST})}{<\lambda - bbm_i(\lambda)>},\ \lambda\in \Lambda, \ \forall\ i & =\\
&&&&&\\
& = & \frac{\mathcal{S}({\rm ST})}{<\mu - bbm_1(\mu)>},\ \mu\in \Lambda^{aug}. &&&\\
\end{array}
$$

Namely, we have:

\begin{thm}\label{dlpl}{{\rm (\cite{DLP, DL4})}}
\[
\begin{array}{llcll}
{\rm i.} & \mathcal{S}\left( L(p,1) \right)& = & \frac{\mathcal{S}({\rm ST})}{<\lambda - bbm_i(\lambda)>}, & \lambda\in \Lambda, \ \forall\ i.\\
&&&&\\
{\rm ii.} & \mathcal{S}\left( L(p,1) \right)& = & \frac{\mathcal{S}({\rm ST})}{<\mu - bbm_1(\mu)>},& \mu\in \Lambda^{aug}.
\end{array}
\]
\end{thm}

\section{Relating the ``positive'' and ``negative'' sub-modules of $\mathcal{S}\left(L(p, 1)\right)$}\label{relposneg}

In this section we relate the infinite system of equations obtained by performing bbm's on elements in $\mathcal{B}^+$ to the infinite system obtained by performing bbm's on elements in $\mathcal{B}^-$. We present now the sets $\mathcal{B}^+$ and $\mathcal{B}^-$ in open braid form:

\begin{equation}\label{Bpn}
\begin{array}{lcl}
\Lambda^{+} & := & \{t_0^{k_0}t_1^{k_1}\ldots t_{n}^{k_n},\ k_i \in \mathbb{N}\backslash \{0\}\ :\ k_{i}\leq k_{i-1},\, \forall\, i\}\\
&&\\
\Lambda^{-} & := & \{t_0^{k_0}t_1^{k_1}\ldots t_{n}^{k_n},\ k_i \in \mathbb{Z}\backslash \mathbb{N}\ :\ k_{i}\leq k_{i-1},\, \forall\, i\}\\
\end{array}
\end{equation}

We now augment the sets $\Lambda^{+}, \Lambda^{-}$ by allowing arbitrary exponents in monomials in the $t_i$'s, and we define the the corresponding subsets of $\Lambda^{+}, \Lambda^{-}$ of level $k$ as follows:

\begin{defn}\label{ponel}\rm
We define the {\it ``positive'' subset} of $\Lambda^{aug}_{(k)}$:
\begin{equation}\label{lamaugp}
\Lambda^{aug_+}_{(k)}\ :=\{t_0^{k_0}t_1^{k_1}\ldots t_{n}^{k_n},\ k_i \in \mathbb{N}\backslash \{0\}\}.
\end{equation}

\noindent and the the {\it ``negative'' subset} of $\Lambda^{aug}_{(k)}$:
\begin{equation}\label{lamaugn}
\Lambda^{aug_-}_{(k)}\ :=\{t_0^{k_0}t_1^{k_1}\ldots t_{n}^{k_n},\ k_i \in \mathbb{Z}\backslash \mathbb{N}\}.
\end{equation}
\end{defn}

\bigbreak

The infinite system of equations obtained by performing $\pm$-$bbm_1$'s on elements in $\Lambda^{aug_+}$ is related to the infinite system of equations obtained by performing $\mp$-$bbm_1$'s on elements in $\Lambda^{aug_-}$ by the following maps:

\begin{defn}\label{mapfI}\rm
\begin{itemize}
\item[(i)] We define the automorphism $f: \Lambda^{aug} \rightarrow \Lambda^{aug}$ such that:
\smallbreak
\[
\begin{array}{rcll}
f(\tau_1\cdot \tau_2) & = & f(\tau_1)\cdot f(\tau_2),& \forall\, \tau_1, \tau_2\in \Lambda^{aug}\\
&&&\\
t_i^k & \mapsto & t_{i}^{-k}, & \forall\ i\in \mathbb{N}^*,\ \forall\ k\in \mathbb{Z} \setminus \mathbb{N}\\\
\sigma_i & \mapsto & \sigma_i^{-1}, & \forall\ i\in \mathbb{N}^*
\end{array}
\]
\bigbreak
\bigbreak

\bigbreak
\item[(ii)] We define the map $I: R[z^{\pm 1}, s_k] \rightarrow R[z^{\pm 1}, s_k]$, $k\in \mathbb{Z}$ such that:
\smallbreak
\[
\begin{array}{rcll}
I(\tau_1 + \tau_2) & = & I(\tau_1) + I(\tau_2), & \forall\, \tau_1, \tau_2\\
I(\tau_1 \cdot \tau_2) & = & I(\tau_1) \cdot I(\tau_2), & \forall\, \tau_1, \tau_2\\
&&&\\
s_{-k} & \mapsto & s_{k}, & \forall\, k \in \mathbb{N}\\
s_{p-k} & \mapsto & s_{p+k}, & \forall\, k\ :\ 0 \leq k\leq p\\
z & \mapsto & \lambda \cdot z &\\
q^{\pm 1} & \mapsto & q^{\mp 1}&\\
\frac{\lambda^k}{z} & \mapsto & \frac{1}{\lambda^{k+1}z}, & \forall\, k\\
\end{array}
\]
\end{itemize}
\end{defn}

\smallbreak

We now state the main result of this paper:

\begin{thm}\label{mthm}
The equations obtained by imposing on the invariant $X$ relations coming from the performance of a $\pm $-$bbm_1$ on an element $\tau$ in $\Lambda^{aug_+}$ are equivalent to the image of the equations obtained by performing $\mp$-$bbm_1$ on its corresponding element $f(\tau)$ in $\Lambda^{aug_-}$ under $I$. That is:
$$I\left(X_{\widehat{f(\tau)}} =  X_{\widehat{bbm_{\mp 1}(f(\tau))}}\right)\ \Leftrightarrow \ X_{\widehat{\tau}} =  X_{\widehat{bbm_{\pm 1}(\tau)}}$$
\end{thm}

Equivalently we have:

\begin{cor}\label{cor1}
The following diagram commutes:
\[
\begin{matrix}
\Lambda^{aug}_{(k)} & \ni & \tau & \overset{bbm_{\pm}}{\rightarrow} & bbm_{\pm 1}(\tau) & \Rightarrow & X_{\widehat{\tau}} =  X_{\widehat{bbm_{1}(\tau)}}, & X_{\widehat{\tau}} = X_{\widehat{bbm_{-1}(\tau)}}\\
                  &     & \updownarrow f &                     &                   &             &  \uparrow I &  \uparrow I\\
\Lambda^{aug}_{(-k)} & \ni & f(\tau) & \overset{bbm_{\mp}}{\rightarrow} & bbm_{\mp 1}(\tau) & \Rightarrow & X_{\widehat{f(\tau)}} =  X_{\widehat{bbm_{-1}(f(\tau))}}, & X_{\widehat{f(\tau)}} = X_{\widehat{bbm_{+1}(f(\tau))}}\\
\end{matrix}
\]
\end{cor}

\begin{defn}\label{sym}\rm
We will say that the elements $bbm_{\pm}{\tau}$ and $bbm_{\mp}{f(\tau)}$ are {\it ``symmetric''} with respect to the sign of their exponents (or just ``symmetric''), although the loop generator $t^p$ appears in both. Moreover, an element in ${\rm H}_n(q)$ will be called {\it braiding ``tail''} and two braiding ``tails'' will be called ``symmetric'' with respect to the sign of their exponents, if one is obtained from the other by changing $\sigma_i^{\pm 1}$ to $\sigma_i^{\mp 1}$.
\end{defn}

\subsection{The infinite system}

Let $\tau_{0,m}^{k_{0,m}} \in \Lambda^{aug}_{(k)}$, that is $\sum_{i=0}^{m}k_i=k$. We present some results on the infinite system of equations (\ref{inft}):

$$
\left\{\begin{matrix}
X_{\widehat{\tau_{0,m}^{k_{0,m}}}}  & = & X_{\widehat{t^{p}\tau_{1,m+1}^{k_{0,m}}\sigma_1}} &\\
&&& \\
X_{\widehat{\tau_{0,m}^{k_{0,m}}}} & = & X_{\widehat{t^{p}\tau_{1,m+1}^{k_{0,m}}\sigma_1^{-1}}} &
\end{matrix}\right.
$$

\noindent which is equivalent to (see Eq.~(\ref{eqpos}) and (\ref{eqneg})):

$$
\left\{\begin{matrix}
tr(\tau_{0,m}^{k_{0,m}})  & = & \frac{1}{z} \cdot \lambda^{\sum_{j=0}^{m}\, (j+1) k_j} \cdot tr({t^{p}\tau_{1,m+1}^{k_{0,m}}g_1}) &\\
&&& \\
tr(\tau_{0,m}^{k_{0,m}}) & = & \frac{1}{z} \cdot \lambda^{\sum_{j=0}^{m}\, (j+1) k_j - 1} \cdot tr(t^{p}\tau_{1,m+1}^{k_{0,m}}g_1^{-1}) &
\end{matrix}\right.
$$

We have the following:

\begin{prop}\label{com}
The unknowns $s_1, s_2, \ldots $ of the system commute.
\end{prop}

\begin{proof}
Consider the set of all permutations of the set $S={k_1,\ldots k_n}$ and let $\varphi$ be a bijection from the set $S$ to itself. We consider now the elements
$\alpha={t^{\prime}_{i_1}}^{k_1}\ldots {t^{\prime}_{i_n}}^{k_n}$ and $\beta={t^{\prime}_{i_1}}^{\varphi(k_1)}\ldots {t^{\prime}_{i_n}}^{\varphi(k_n)}$, where $0\leq i_1\leq i_2 \leq \ldots \leq i_n$ of the basis of $S(ST)$.
We have that: $tr(\alpha)=s_{k_n}\ldots s_{k_1}$ and $tr(\beta)=s_{\varphi(k_n)}\ldots s_{\varphi(k_1)}$.
We compute the invariant $X$ on the closures $\widehat{\alpha}, \widehat{\beta}$ of $\alpha$ and $\beta$, respectively, and we obtain:
$X_{(\widehat{\alpha})}=[-\frac{1-\lambda q}{\sqrt{\lambda}}]^{n-1}\sqrt{\lambda}^{0}tr(\alpha)=[-\frac{1-\lambda q}{\sqrt{\lambda}}]^{n-1}s_{k_n}\ldots s_{k_1}$ and
$X_{(\widehat{\beta})}=[-\frac{1-\lambda q}{\sqrt{\lambda}}]^{n-1}\sqrt{\lambda}^{0}tr(\beta)=[-\frac{1-\lambda q}{\sqrt{\lambda}}]^{n-1}s_{\varphi(k_n)}\ldots s_{\varphi(k_1)}$.
Now, the $n$-component link $\widehat{\alpha}$ is isotopic to $\widehat{\beta}$ in $ST$, as illustrated in Figure~\ref{pic12} for the case of two components.
So, we have that $X_{(\widehat{\alpha})}=X_{(\widehat{\beta})}$, equivalently,
\begin{equation}\label{unknown}
s_{k_n}\ldots s_{k_1}=s_{\varphi(k_n)}\ldots s_{\varphi(k_1)}
\end{equation}
and so the unknowns of the system commute.\\
Equation~\ref{unknown} holds for any subset $S$ of $\mathbb{Z}$ and for any permutation $\phi$ of $S$, hence the unknowns $s_i$ of the system
 $(\ref{inft})$ must commute.
\end{proof}

\begin{figure}
\begin{center}
\includegraphics[width=4.5in]{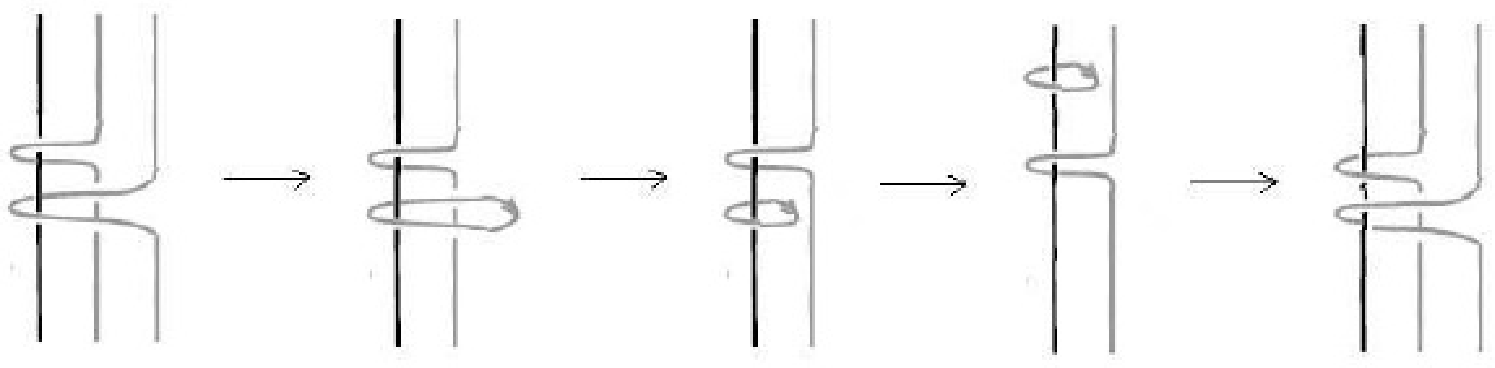}
\end{center}
\caption{ $t^{-1}t_1^{\prime}\ = \ t{t_1^{\prime}}^{-1}$. }
\label{pic12}
\end{figure}

\subsection{The modules $\frac{\Lambda^{aug_+}}{<{\rm bbm_1}>}$ and $\frac{\Lambda^{aug_-}}{<{\rm bbm_1}>}$}\label{msec}

From now on, we will consider all braid band moves to take place on the first moving strand of elements in $\Lambda^{aug}$ and we will denote by $bbm_+(\tau)$ the result of the performance of a positive $bbm_1$ and $bbm_-(\tau)$ will correspond to the result of the performance of a negative $bbm_1$ on $\tau\in \Lambda^{aug}$.

\smallbreak

Let $\tau:= \tau_{0, n}^{k_{0, n}}\in \Lambda^{aug}_{+}$ and perform a $bbm_+$. We have that 

$$bbm_+(\tau)=t^p\tau_{1, n+1}^{k_{0, n}}\sigma_1$$

\noindent and we obtain the equation:

\[
X_{\widehat{\tau}}=X_{\widehat{bbm_+(\tau)}}\ \Leftrightarrow\ tr(\tau_{0, n}^{k_{0, n}})=-\frac{1-\lambda q}{\sqrt{\lambda}(1-q)}\, \sqrt{\lambda}^{\, 1+\underset{i=0}{\overset{n}{\sum}}\, 2(i+1)k_i}\cdot tr(t^p \tau_{1, n+1}^{k_0, n}g_1)
\]

\noindent and since $\lambda=\frac{z+1-q}{qz}$, we obtain

\begin{equation}\label{eqpos}
tr(\tau_{0, n}^{k_{0, n}})=\frac{\lambda^{\underset{i=0}{\overset{n}{\sum}}\, (i+1)k_i}}{z}\cdot tr(t^p \tau_{1, n+1}^{k_0, n}g_1)
\end{equation}

Consider now $f(\tau):= \tau_{0, n}^{-k_{0, n}}\in \Lambda^{aug}_{-}$ and perform a $bbm_-$. We have that 

$$bbm_-(f(\tau))=t^p\tau_{1, n+1}^{-k_{0, n}}\sigma_1^{-1}$$

\noindent and we obtain the equation:

\[
X_{\widehat{f(\tau)}}=X_{\widehat{bbm_-(f(\tau))}}\ \Leftrightarrow\ tr(\tau_{0, n}^{-k_{0, n}})=-\frac{1-\lambda q}{\sqrt{\lambda}(1-q)}\, \sqrt{\lambda}^{\, 1-\underset{i=0}{\overset{n}{\sum}}\, 2(i+1)k_i}\cdot tr(t^p \tau_{1, n+1}^{k_0, n}g_1)
\]

\noindent , that is:

\begin{equation}\label{eqneg}
tr(\tau_{0, n}^{-k_{0, n}})=\frac{\lambda^{-1-\underset{i=0}{\overset{n}{\sum}}\, (i+1)k_i}}{z}\cdot tr(t^p \tau_{1, n+1}^{-k_0, n}g_1^{-1})
\end{equation}

In order to prove Corollary~\ref{cor1}, we first prove that the image of the coefficient in Equation~\ref{eqneg} under the map $I$ is equal to the coefficient of Equation~\ref{eqpos}. Indeed we have the following:

\begin{lemma}\label{coef}
\[
I\left( \frac{\lambda^{-1-\underset{i=0}{\overset{n}{\sum}}\, (i+1)k_i}}{z} \right)\ =\ \frac{\lambda^{\underset{i=0}{\overset{n}{\sum}}\, (i+1)k_i}}{z}
\]
\end{lemma}

\begin{proof}
We have that:
\[
I\left( \frac{\lambda^{-1-\underset{i=0}{\overset{n}{\sum}}\, (i+1)k_i}}{z} \right)\ =\ \frac{\lambda^{1+\underset{i=0}{\overset{n}{\sum}}\, (i+1)k_i -1}}{z} \ =\ \frac{\lambda^{\underset{i=0}{\overset{n}{\sum}}\, (i+1)k_i}}{z}
\]
\end{proof}

\begin{remark}\label{rmk2}\rm
Lemma~\ref{coef} demonstrates the motivation for the Definition~\ref{mapfI}(ii) on $\frac{\lambda^k}{z}$, where $\frac{\lambda^k}{z}\overset{I}{\mapsto} \frac{1}{\lambda^{k+1}\, z}$.
\end{remark}

We now relate $tr(\tau)$ to $I\left(tr\left(f(\tau)\right)\right)$ using the fact that relations used in order to convert elements in $\Lambda^{aug_+}$ to sums of elements in $\left(\Lambda^{\prime}\right)^{aug_+}$, where $\left(\Lambda^{\prime}\right)^{aug_+}$ is defined as monomials in $t_i^{\prime}$'s with positive exponents, are ``symmetric'', as shown for example below:

\[
\begin{array}{lclclcl}
\sigma_i & = & q\, \sigma_i^{-1}\ +\ (q^{-1}-1) & \& & \sigma_i^2 & = & (q-1)\, \sigma_i\ +\ q\\
&&&&&&\\
\sigma_i^{-1} & = & q^{-1}\, \sigma_i\ +\ (q-1) & \& & \sigma_i^{-2} & = & (q^{-1}-1)\, \sigma_i^{-1}\ +\ q^{-1}\\
\end{array}
\]

\noindent Moreover, these relations lead to $q \overset{I}{\mapsto} q^{-1}$ in Definition~\ref{mapfI}(ii), and the rule $z \overset{I}{\mapsto} z\cdot \lambda$ comes from the fact that $tr(\sigma_i)\ =\ z$, while $tr(\sigma_i^{-1})\ =\ q^{-1}z+(q^{-1}-1)\ = \ \frac{z+1-q}{q}\ =\ \lambda\cdot z$. The reader is now referred to \cite{La1, La2, DL2, DL3, DL4, DLP} for other ``symmetric'' relations, and also for more details on the techniques applied in order to obtain the infinite change of basis matrix relating the sets $\Lambda$ and $\Lambda^{\prime}$.

\smallbreak

We are now in position to translate Theorems~\ref{convert}, \ref{tails} and \ref{exp} in the context of this paper:

\bigbreak

\begin{itemize}
\item[$\bullet$] Theorem~\ref{convert} suggests that an element $\tau$ in $\Lambda^{aug_+}$ can be written as a sum of its homologous word $\tau^{\prime}$ in $(\Lambda^{\prime})^{aug_+}$, the element $\tau$ followed by a braiding ``tail'' and monomials in $\Sigma$ which are of lower order than the initial monomial $\tau$.
\smallbreak
Similarly, an element $T$ in $\Lambda^{aug_-}$ can be written as a sum of its homologous word $T^{\prime}$ in $(\Lambda^{\prime})^{aug_-}$, the element $T$ followed by a braiding ``tail'' and monomials in $\Sigma$ which are of lower order than the initial monomial $T$.
\smallbreak
Moreover, as explained above, corresponding coefficients in these two processes will be ``symmetric'' and the braiding ``tails'' in $\tau$, after Theorem~\ref{convert} is applied, will be ``symmetric'' to the braiding ``tails'' in $T$.

\bigbreak

\item[$\bullet$] For the elements $\tau$ and $T$ that are followed by braiding ``tails'' in ${\rm H}_n(q)$, we apply Theorem~\ref{tails} and we have that corresponding coefficients will be ``symmetric'' only for the terms involving the parameter $q$, while for $z$, the reader is referred to the discussion after Remark~\ref{rmk2}.

\bigbreak

\item[$\bullet$] By Theorem~\ref{exp} the order of the exponents in an element in $\Lambda^{aug}$ can be altered, leading to elements of lower (or even greater) order. For the braiding ``tail'' occurring after applying Theorem~\ref{exp}, we apply Theorem~\ref{tails} again, and this procedure will eventually stop and the result will be a sum of elements in $\Lambda^{aug}$ of lower order than the initial element. For more details of how this procedure terminates the reader is referred to \cite{DL2} and \cite{DL3}.
\end{itemize}

We now have the following result:

\begin{prop}\label{trtau}
For $\tau\in \Lambda^{aug_+}_{(k)}$, where $k\in \mathbb{N}$, the following relation holds:
\[
tr(\tau)\ =\ I\left( tr\left( f(\tau) \right) \right),
\]
\noindent where $f(\tau)\in \Lambda^{aug_-}_{(-k)}$.
\end{prop}

\begin{proof}
Let $\tau:=\tau_{0, n}^{k_{0, n}}\in \Lambda^{aug_+}_{(k)}$, $k\in \mathbb{N}$ and $f(\tau):=\tau_{0, n}^{-k_{0, n}}\in \Lambda^{aug_-}_{(-k)}$. In order to evaluate $tr(\tau)$ we use the inverse of the change of basis matrix and express $\tau$ as a sum of elements in $\left(\Lambda^{\prime}\right)^{aug_+}_{(k)}$, i.e. $\tau \ \widehat{\cong} \ \underset{i}{\sum}\, A_i\, \tau_i^{\prime}$, where $A_i$ coefficients in $\mathbb{C}[q^{\pm 1}, z^{\pm 1}]$ for all $i$, such that $\exists\, j\ :\ \tau_j^{\prime}=\tau^{\prime} \sim \tau$ and $\tau_i^{\prime}<\tau^{\prime}$, for all $i \neq j$. Following the same steps in order to express $f(\tau)$ as a sum of elements in $\left(\Lambda^{\prime}\right)^{aug_-}_{(-k)}$, we prove that we obtain that $f(\tau)\ \widehat{\cong}\ \underset{i}{\sum}\, B_i\, f(\tau_i^{\prime})$, where $B_i$ coefficients in $\mathbb{C}[q^{\pm 1}, z^{\pm 1}]$ for all $i$, such that $\exists\, j\ :\ f(\tau_j^{\prime})=f(\tau^{\prime}) \sim f(\tau)$ and $f(\tau_i^{\prime})<f(\tau^{\prime})$, for all $i \neq j$, and also that $f(B_i)\ =\ A_i$, for all $i$. We prove that by strong induction on the order of $\tau$.

\smallbreak

The base of induction is $t^k\in \Lambda^{aug_+}_{(k)}$, where $tr(t^k)=s_k$ and $f(t^k)=t^{-k}$ and $tr(t^{-k})=s_{-k}$. We observe that $s_{-k}\overset{I}{\mapsto}s_k$ by Definition~\ref{mapfI}, and thus $I\left(tr\left( f(t^{-k}) \right) \right)\ =\ tr\left(t^k \right)$.

\smallbreak

Assume now that $I\left( tr\left( f(\tau_i) \right)\right) \ =\ tr(\tau_i)$, for all $\tau_i<\tau \in \Lambda^{aug_+}_{(k)}$. Then, for $\tau$ we have that:

\smallbreak

\[
\begin{array}{lcl}
\tau & := & \tau_{0, n}^{k_{0, n}} \ := \ t^{k_0}t_1^{k_1}\ldots t_n^{k_n}    \ = \ t^{k_0}\ldots t_n^{k_n-1}\cdot (\sigma_n \ldots \sigma_1\, t\, \underline{\sigma_1} \ldots \sigma_n) \ = \\  
&&\\
     & = & \underset{A}{\underbrace{q\, \tau_{0, n-1}^{k_{0, n-1}}\cdot t_n^{k_n-1} \cdot (\sigma_n\ldots \sigma_1\, t\, \sigma_1^{-1}\sigma_2\ldots \sigma_n)}} \ + \\
&&\\
     & + & \underset{B}{\underbrace{(q-1)\, t^{k_0+1}\cdot \tau_{1, n-1}^{k_{1, n-1}}\cdot t_n^{k_n-1}\cdot (\sigma_n \ldots \sigma_1 \sigma_2\ldots \sigma_n)}} \\
\end{array}
\]

\noindent and for $f(\tau)$ we obtain:

\[
\begin{array}{lcl}
f(\tau)  & := & \tau_{0, n}^{-k_{0, n}}\ := \ t^{-k_0}t_1^{-k_1}\ldots t_n^{-k_n} \ = \ t^{-k_0}\ldots t_n^{-k_n+1}\cdot (\sigma_n^{-1} \ldots \underline{\sigma_1^{-1}}\, t^{-1}\, \sigma_1^{-1} \ldots \sigma_n^{-1}) \ =\\  
&&\\
         & = & \underset{C}{\underbrace{q^{-1}\, \tau_{0, n-1}^{-k_{0, n-1}}\cdot t_n^{-k_n+1} \cdot (\sigma_n^{-1}\ldots \sigma_2^{-1}\sigma_1\, t\, \sigma_1^{-1}\ldots \sigma_n^{-1})}} \ + \\
&&\\
         & + & \underset{D}{\underbrace{(q^{-1}-1)\, t^{-k_0+1}\cdot \tau_{1, n-1}^{-k_{1, n-1}}\cdot t_n^{-k_n+1}\cdot (\sigma_n^{-1} \ldots \sigma_1^{-1} \ldots \sigma_n^{-1})}}\\
\end{array}
\]

Observe now that $D\ =\ f(B)$ and according to Definition~\ref{order} we have that

$$t^{k_0+1}\cdot \tau_{1, n-1}^{k_{1, n-1}}\cdot t_n^{k_n-1}\cdot (\sigma_n \ldots \sigma_1 \sigma_2\ldots \sigma_n)\ <\ \tau.$$ 

\noindent Thus, from the induction step we obtain that $I\left(tr(D)\right)\ =\ tr(B)$.

\smallbreak

We now apply Theorems~7, 8, 9 \& 10 in \cite{DL2} (relations used for the change of basis matrix and which are Theorems~\ref{convert}, \ref{tails}, \ref{exp} \& Eq.(\ref{gaps}) in this paper) on $A$ and $C$ and we use the induction step to reduce the complexity of the relations obtained each time a lower order term appears in the relations. As explained in \cite{DL2, DL4}, we will eventually obtain the homologous word in $\left(\Lambda^{\prime} \right)^{aug_+}_{k}$ for $\tau$ (coming from $A$) and the homologous word in $\left(\Lambda^{\prime} \right)^{aug_-}_{-k}$ for $f(\tau)$ (coming from $C$), with ``symmetric'' coefficients. Hence, 

\[
tr(\tau)\ =\ I\left( tr\left( f(\tau) \right) \right).
\]

\end{proof}

In order to proceed with the proof of Corolarry~\ref{cor1}, we will need the following relations (Lemma~2 \cite{La2}):

\begin{equation}\label{sllem}
\begin{array}{lclcll}
t_n^k\, \sigma_n & = & (q-1)\, \underset{j=0}{\overset{k-1}{\sum}}\, q^j\, t_{n-1}^j t_n^{k-j} & + & q^k\, \sigma_n\, t_{n-1}^k &,\ k\in \mathbb{N}\\
&&&&&\\
t_n^{-k}\, \sigma_n^{-1} & = & (q^{-1}-1)\, \underset{j=0}{\overset{-k+1}{\sum}}\, q^j\, t_{n-1}^j t_n^{-k-j} & + & q^{-k}\, \sigma_n\, t_{n-1}^{-k}\ \sigma_n^{-1} &,\ k\in \mathbb{N}
\end{array}
\end{equation}

\noindent and the following lemmas:
 
\begin{lemma}\label{imp}
For $n, k\in \mathbb{N}$, the following relations hold:
\[
\begin{array}{llcl}
{\rm i.} & t_n^{k} & = & \underset{j=1}{\overset{k-1}{\sum}}\, q^{j-1}\, (q-1)\, t_{n-1}^{j}\, t_n^{k-j}\cdot \sigma_n\ +\ q^{k-1}\, \sigma_n\, t_{n-1}^k\, \sigma_n\\
&&&\\
{\rm ii.} & t_n^{-k} & = & \underset{j=-1}{\overset{-k+1}{\sum}}\, q^{j+1}\, (q^{-1}-1)\, t_{n-1}^{j}\, t_n^{-k-j}\cdot \sigma_n^{-1}\ +\ q^{-k+1}\, \sigma_n^{-1}\, t_{n-1}^{-k}\, \sigma_n^{-1}\\
\end{array}
\]
\end{lemma}

\begin{proof}
We only prove relations (i). Relations (ii) follow similarly. For $k=2$ we have that:

\[
\begin{array}{lcl}
t_n^2 & = & \sigma_{n}\, t_{n-1}\, \underline{\sigma_n^2}\, t_{n-1}\, \sigma_n \ = \ (q-1)\, \underline{\sigma_{n}\, t_{n-1}\, \sigma_n}\, t_{n-1}\, \sigma_n\ +\ q\, \sigma_n\, t_{n-1}^2\, \sigma_n\ =\\
&&\\
& = & (q-1)\, t_{n-1}\, \, t_{n}\, \sigma_n\ +\ q\, \sigma_n\, t_{n-1}^2\, \sigma_n \\
\end{array}
\]

For $k\in \mathbb{N}$ we have:

\[
\begin{array}{lcl}
t_n^k\ =\ t_n^{k-2}\cdot \underline{t_n^2} & = & (q-1)\, t_n^{k-1} t_{n-1}\, \sigma_n\ +\ q\, \underline{t_n^{k-2} \sigma_n}\, t_{n-1}^2\, \sigma_n\ =\\
&&\\
 & \overset{Eq.~\ref{sllem}}{=} & (q-1)\, t_n^{k-1} t_{n-1}\, \sigma_n\ +\ q\, (q-1)\, \underset{j=0}{\overset{k-3}{\sum}}\, q^j\, t_{n-1}^j t_n^{k-2-j}\, t_{n-1}^2\, \sigma_n\ + \\
&&\\
& + & q^{k-1}\, \sigma_n\, t_{n-1}^{k-2}\, t_{n-1}^2\, \sigma_n\ = \\
&&\\
& = & \underset{j=1}{\overset{k-1}{\sum}}\, q^{j-1}\, (q-1)\, t_{n-1}^{j}\, t_n^{k-j}\cdot \sigma_n\ +\ q^{k-1}\, \sigma_n\, t_{n-1}^k\, \sigma_n\\
\end{array}
\]

\end{proof}

\begin{lemma}\label{llem}
For $n, k\in \mathbb{N}$, the following relations hold:
\[
\begin{array}{llcl}
{\rm i.} & t_n^{k}\, \sigma_{n+1} & = & q^{-k+1}\, \sigma_{n+1}^{-1}\, t_{n+1}^k\ +\ \underset{j=1}{\overset{k-1}{\sum}}\, q^{-j+1}\, (q^{-1}-1)\, t_{n}^{k-j}\, t_{n+1}^{j}\\
&&&\\
{\rm ii.} & t_n^{-k}\, \sigma_{n+1}^{-1} & = & q^{k-1}\, \sigma_{n+1}\, t_{n+1}^{-k}\ +\ \underset{j=1}{\overset{k-1}{\sum}}\, q^{j-1}\, (q-1)\, t_{n}^{-k+j}\, t_{n+1}^{-j}\\
\end{array}
\]
\end{lemma}

\begin{proof}
We prove relations (i) by induction on $k\in \mathbb{N}$. Relations (ii) follow similarly.
\smallbreak
For $k=1$ we have that $t_n\, \sigma_{n+1}\ =\ \sigma_{n+1}^{-1}\, t_{n+1}$, which is true. Assume now that the relation holds for $k-1$. Then, for $k$ we have:
\[
\begin{array}{lcl}
t_n^{k}\, \sigma_{n+1} & = & t_n\, \underline{t_n^{k-1}\, \sigma_{n+1}}\ \overset{ind.}{\underset{step}{=}} \\
&&\\
& = & q^{-k+2}\, t_n\, \underline{\sigma_{n+1}^{-1}}\, t_{n+1}^{k-1}\ +\ \underset{j=1}{\overset{k-2}{\sum}}\, q^{-j+1}\, (q^{-1}-1)\, t_{n}^{k-j}\, t_{n+1}^{j}\ =\\
&&\\
& = & q^{-k+1}\, \underline{t_n\, \sigma_{n+1}}\, t_{n+1}^{k-1}\ +\ q^{-k+2}(q^{-1}-1)\, t_n\, t_{n+1}^{k-1}+\ \underset{j=1}{\overset{k-2}{\sum}}\, q^{-j+1}\, (q^{-1}-1)\, t_{n}^{k-j}\, t_{n+1}^{j}\ =\\
&&\\
& = & q^{-k+1}\, \sigma_{n+1}^{-1}\, t_{n+1}^k\ +\ \underset{j=1}{\overset{k-1}{\sum}}\, q^{-j+1}\, (q^{-1}-1)\, t_{n}^{k-j}\, t_{n+1}^{j}\\
\end{array}
\]
\end{proof}

We are now in position to prove the following lemma that serves as the basis of the induction applied in the final result of this section, Proposition~\ref{fl}, which will conclude the proof of Corollary~\ref{cor1}.

\begin{lemma}\label{basind}
Let $t^k\ \overset{bbm_{\pm}}{\longrightarrow}\ t^pt_1^k\, \sigma^{\pm 1}$ and $f(t^k):=t^{-k}\ \overset{bbm_{\mp}}{\longrightarrow}\ t^pt_1^{-k}\, \sigma^{\mp 1}$. Then, the following relations hold:
\[
I\left( tr(t^pt_1^{-k}\, \sigma_1^{\mp 1}) \right) \ = \ tr(t^pt_1^{k}\, \sigma_1^{\pm 1}).
\]
\end{lemma}

\begin{proof}
We only prove that $I\left( tr(bbm_-(f(t^k))) \right) \ = \ tr(bbm_+(t^{k}))$, by strong induction on the order of $bbm_-(f(t^k))$. The case $I\left( tr(bbm_+(f(t^k))) \right) \ = \ tr(bbm_-(t^{k}))$ follows similarly.

\bigbreak

The base of induction is the case $k=1$, i.e. $I\left( tr(t^pt_1^{-1}\, \sigma_1^{-1}) \right) \ = \ tr(t^pt_1\, \sigma_1)$. We have that:

\[
\begin{array}{lclclclc}
tr(t^p \underline{t_1\, \sigma_1}) & = & (q-1)\, tr(t^p\underline{t_1}) & + & q\, tr(t^p\, \sigma_1\, \underline{t}) & = & \\
&&&&&&\\
                                   & = & q(q-1)\, tr(t^pt_1^{\prime})    & + & (q-1)^2\, tr(t^{p+1}\, \sigma_1) & + & q\, tr(t^{p+1}\, \sigma_1) & =\\ 
&&&&&&\\
 & = & q(q-1)\, s_1s_p & + & (q-1)^2\, z\, s_{p+1} & + & qz\, s_{p+1}&\\																
\end{array}
\]

\noindent and

\[
\begin{array}{lclclclc}
tr(t^p \underline{t_1^{-1}\, \sigma_1^{-1}}) & = & (q^{-1}-1)\, tr(t^p\underline{t_1^{-1}}) & + & q^{-1}\, tr(t^p\, \sigma_1^{-1}\, \underline{t^{-1}}) & = & \\
&&&&&&\\
                                   & = & q^{-1}(q^{-1}-1)\, tr(t^p{t_1^{\prime}}^{-1})    & + & (q^{-1}-1)^2\, tr(t^{p-1}\, \sigma_1^{-1}) & + &  & \\ 
&&&&&&\\
 & + &		q^{-1}\, tr(t^{p-1}\, \sigma_1^{-1})		& = &&&\\																			
&&&&&&\\
 & = & q^{-1}(q^{-1}-1)\, s_{-1}s_p & + & q^{-1}(q^{-1}-1)^2\, z\, s_{p-1} & + &  & \\
&&&&&&\\
& + & (q^{-1}-1)^3\, s_{p-1}& + & q^{-1}(q^{-1}-1)\, s_{p-1} & + & q^{-2}z\, s_{p-1}  &\\
\end{array}
\]

Moreover:
\[
\begin{array}{lclclc}
I\left( tr(t^p \underline{t_1^{-1}\, \sigma_1^{-1}})  \right) & = & I\left(q^{-1}(q^{-1}-1)\, s_{-1}s_p \right) & + & I\left(q^{-1}(q^{-1}-1)^2\, z\, s_{p-1}\right) & + \\
&&&&&\\
& + & I\left( (q^{-1}-1)^3\, s_{p-1} \right) & + &  I\left( q^{-1}(q^{-1}-1)\, s_{p-1} \right) &+\\
&&&&&\\
& + & I\left( q^{-2}z\, s_{p-1}\right) &= & &\\
&&&&&\\
& = & q(q-1)\, s_1s_p & + & q(q-1)^2\, \lambda z\, s_{p+1} & +\\
&&&&&\\
& + & (q-1)^3\, s_{p+1} & + & q^2\, z \lambda\, s_{p+1} & +\\
&&&&&\\
& + & q(q-)\, s_{p+1} & \Rightarrow &&\\
\end{array}
\]

\[
\begin{array}{lcl}
I\left( tr(t^p t_1^{-1}\, \sigma_1^{-1})  \right) & = & q(q-1)\, s_{1}s_p \ + \\
&&\\
 & + & \left[q(q-1)^2\, \frac{z+1-q}{qz}\ + \ (q-1)^3 \ +\ q^2z\, \frac{z+1-q}{qz}\ +\ q(q-1) \right]\, s_{p+1}\ =\\
&&\\
 & = & q(q-1)\, s_1s_p \ + \ (q-1)^2\, z\, s_{p+1} \ + \ qz\, s_{p+1} \ =\\
&&\\
& = & tr(t^pt_1^k\, \sigma_1) \\
\end{array}
\]

\bigbreak

Assume now that $I\left( tr(t^{p-i}t_1^{-k-i}\, \sigma_1^{-1}) \right) \ = \ tr(t^{p+i}t_1^{k-i}\, \sigma_1)$, for all $0<i<k$. Then, for $i=0$ we have:

\[
\begin{array}{lcl}
tr\left(t^p t_1^k\, \sigma_1 \right) & \overset{Eq.~\ref{sllem}}{=} & (q-1)\, \underset{j=0}{\overset{k-1}{\sum}}\, q^j\, t^{p+j} t_1^{k-j}\ +\ q^k\, z\, tr(t^{p+k})\ =\\
&&\\
 & = & (q-1)\, \underset{j=0}{\overset{k-1}{\sum}}\, q^j\, tr(t^{p+j} t_1^{k-j})\ +\ q^k\, z\, s_{p+k}\\
&&\\
tr\left(t^p t_1^{-k}\, \sigma_1^{-1} \right) & \overset{Eq.~\ref{sllem}}{=} & (q^{-1}-1)\, \underset{j=0}{\overset{-k+1}{\sum}}\, q^j\, tr(t^{p+j} t_1^{-k-j}) \ + \\ 
&&\\
& + & \ q^{-k-1}\, z\, tr(t^{p-k})\ +\ q^{-k}\, (q^{-1}-1)\, tr(t^{p-k})\ = \\
&&\\
 & = & (q^{-1}-1)\, \underset{j=0}{\overset{-k+1}{\sum}}\, q^j\, tr(t^{p+j} t_1^{-k-j}) \ + \ \left(q^{-k-1}\, z\, \ +\ q^{-k}\, (q^{-1}-1)\right)\, s_{p-k}
\end{array}
\]

We have that 
\[
\begin{array}{lcl}
I\left(\left(q^{-k-1}\, z \ +\ q^{-k}\, (q^{-1}-1)\right)\, s_{p-k}\right) & = & \left(q^{k+1}\, \lambda\, z\ +\ q^k\, (q-1)\right)\, s_{p+k}\ =\\
&&\\
& = & \left(q^{k+1}\, \frac{z+1-q}{qz}\, z\ +\ q^k\, (q-1)\right) \, s_{p+k} \ =\ q^k\, z\, s_{p+k}\\
\end{array}
\]

\smallbreak

Moreover, the terms $t^{p+j} t_1^{k-j}$ are of lower order than $t^p t_1^{k}$ for all $j\neq 0$ and thus, from the induction step we have that:
\[
I\left( tr(t^{p+j} t_1^{k-j}) \right)\ =\ tr(t^{p+j} t_1^{-k-j}),\ {\rm for\ all}\ j\neq 0.
\]

\smallbreak

For $j=0$ we have:

\[
\begin{array}{lclc}
t^pt_1^k & = &t^pt_1^{k-2}\, \underline{t_1^2} \ = \ (q-1)\, t^{p+1}t_1^{k-1}\, \sigma_1\ +\ q\, t^p\, \underline{t_1^{k-2}\, \sigma_1}\, t^2\sigma_1 & = \\
&&&\\
& = & (q-1)\, t^{p+1}t_1^{k-1}\, \sigma_1\ +\ q\, (q-1)\, \underset{j=0}{\overset{k-3}{\sum}}\, q^j t^{p+j+2}\,t_1^{k-2-j}\, \sigma_1\ +\ q^{k-2}\, t^p \sigma_1t^k\underline{\sigma_1} & = \\
&&&\\
& = & \underset{A}{\underbrace{(q-1)\, t^{p+1}t_1^{k-1}\, \sigma_1}}\ +\ \underset{B}{\underbrace{q\, (q-1)\, \underset{j=0}{\overset{k-3}{\sum}}\, q^j t^{p+j+2}\, t_1^{k-2-j}\, \sigma_1}} &+\\
&&&\\
& + & \underset{C}{\underbrace{q^{k-1}\, t^p {t_1^{\prime}}^k}}\ +\ \underset{D}{\underbrace{q^{k-2}\, (q-1)\, t^p\sigma_1t^k}} &  \\
\end{array}
\]

and

\[
\begin{array}{lclc}
t^pt_1^{-k} & = & t^pt_1^{-k+2}\, \underline{t_1^{-2}} \ = \ (q^{-1}-1)\, t^{p-1}\, t_1^{-k+1}\, \sigma_1^{-1}\ +\ q^{-1}\, t^p\, \underline{t_1^{-k+2}\, \sigma_1^{-1}}\, t^{-2}\sigma_1^{-1} & =\\
&&&\\
& = & (q^{-1}-1)\, t^{p-1}t_1^{-k+1}\, \sigma_1^{-1}\ +\ q\, (q-1)\, \underset{j=0}{\overset{-k+3}{\sum}}\, q^j t^{p+j-2}\,t_1^{-k+2-j}\, \sigma_1^{-1} & +\\
&&&\\
& + & q^{-k+2}\, t^p\, \underline{\sigma_1^{-1}}\, t^{-k}\sigma_1^{-1} & = \\
&&&\\
& = & \underset{A^{\prime}}{\underbrace{(q^{-1}-1)\, t^{p-1}t_1^{-k+1}\, \sigma_1^{-1}}}\ +\ \underset{B^{\prime}}{\underbrace{q\, (q-1)\, \underset{j=0}{\overset{-k+3}{\sum}}\, q^j t^{p+j-2}\,t_1^{-k+2-j}\, \sigma_1^{-1}}} & +\\
&&&\\
& + & \underset{C^{\prime}}{\underbrace{q^{-k+1}\, t^p {t_1^{\prime}}^{-k}}} \ +\ \underset{D^{\prime}}{\underbrace{q^{-k+2}\, (q^{-1}-1)\, t^p t^{-k} \sigma_1^{-1}}}    & \\
\end{array}
\]

Observe now now that in $A$ the term $t^{p+1}t_1^{k-1}\, \sigma_1$ is of lower order than $t^pt_1^k$ and also that the terms $A, A^{\prime}$ are ``symmetric''. From the induction step we have that $I\left(tr(A^{\prime})\right)\ =\ tr(A)$. For the same reasons $I\left(tr(B^{\prime})\right)\ =\ tr(B)$. Finally we have that:

\[
\begin{array}{lclclr}
tr(C) & = & q^{k-1}\, tr(t^p {t_1^{\prime}}^k) & = & q^{k-1}\, s_p\, s_k & \\
&&&&  & \Rightarrow\ I\left(q^{-k+1}\, s_p\, s_{-k}\right)\ =\ q^{k-1}\, s_p\, s_k \\
tr(C^{\prime}) & = & q^{-k+1}\, t^p {t_1^{\prime}}^{-k} & = & q^{-k+1}\, s_p\, s_{-k} &\\
\end{array}
\]

\noindent and that

\[
\begin{array}{lclcl}
tr(D) & = & q^{k-2}\, (q-1)\, tr(t^{p+k}\sigma_1) & = & q^{k-2}\, (q-1)\, z\, s_{p+k} \\
&&&& \\
tr(D^{\prime}) & = & q^{-k+2}\, (q^{-1}-1)\, tr(t^{p-k}\sigma_1^{-1}) & = & q^{-k+1}\, (q^{-1}-1)\, z\, s_{p-k}\ +\ q^{-k+2}\, (q^{-1}-1)^2\, s_{p-k}\\
\end{array}
\]

\[
\begin{array}{lcl}
I\left( tr(D^{\prime}) \right) & = & \left(q^{k-1}\, (q-1)\, \lambda\, z\ +\ q^{k-2}\, (q-1)^2 \right)\, s_{p-k}\\
&&\\
& = & \left(q^{k-1}\, (q-1)\, \frac{z+1-q}{qz}\, z\ +\ q^{k-2}\, (q-1)^2 \right)\, s_{p-k}\\
&&\\
& = & \left(q^{k-2}\, (q-1)\, (z+1-q)\ +\ q^{k-2}\, (q-1)^2 \right)\, s_{p-k}\\
&&\\
& = & q^{k-2}\, (q-1)\, z\, s_{p+k}\ \Rightarrow\\
&&\\
I\left( tr(D^{\prime}) \right) & = & tr(D)\\
\end{array}
\]

The proof is now concluded.
\end{proof}

We are now ready to state and prove the final proposition that concludes the proof of Corollary~\ref{cor1}.

\begin{prop}\label{fl}
Let $\tau\in \Lambda^{aug_+}_{(k)}$, where $k\in \mathbb{N}$, and $bbm_+\left(\tau\right)$ the result of the performance of a positive braid band move on $\tau$. Let also $f(\tau)\in \Lambda^{aug_-}_{(-k)}$ and $bbm_-\left(f(\tau)\right)$ the result of the performance of a negative braid band move on $f(\tau)$. Then, the following relation holds:
\[
I\left( tr\left( bbm_{\mp}\left(f(\tau) \right)\right) \right) \ = \ tr\left(bbm_{\pm }(\tau)\right),\ {\rm where}\ \tau\in \Lambda^{aug_+}_{(k)}.
\]
\end{prop}

\begin{proof}
We only prove that $I\left( tr\left( bbm_{-}\left(f(\tau) \right)\right) \right) \ = \ tr\left(bbm_{+}(\tau)\right)$, by strong induction on the order of $bbm_{-}\left(f(\tau) \right)$. The case $I\left( tr\left( bbm_{+}\left(f(\tau) \right)\right) \right) \ = \ tr\left(bbm_{-}(\tau)\right)$ follows similarly. Let

\[
\begin{array}{rclrcl}
\tau & = & t^{k_0} t_1^{k_1}\ldots t_n^{k_n}, & f(\tau) & = & t^{-k_0} t_1^{-k_1}\ldots t_n^{-k_n},\\
&&&&\\
bbm_{+}(\tau) & = & t^{p} t_1^{k_0}\ldots t_n^{k_{n-1}}t_{n+1}^{k_n}\, \sigma_1, & bbm_{-}(f(\tau)) & = & t^{p} t_1^{-k_0}\ldots t_n^{-k_{n-1}}t_{n+1}^{-k_n}\, \sigma_1^{-1}.\\
\end{array}
\]

\smallbreak

The base of induction is Lemma~\ref{basind}. Assume now that $I\left( tr\left( bbm_{-}\left(f(\tau_i) \right)\right) \right) \ = \ tr\left(bbm_{+}(\tau_i)\right)$, for all $\tau_i<\tau$. Then, for $\tau$ we have that: 

\[
\begin{array}{lcl}
tr(t^{p} t_1^{k_0}\underline{\ldots t_n^{k_{n-1}}t_{n+1}^{k_n}\, \sigma_1}) & = & tr(t^{p}\, (\underline{t_1^{k_0}\cdot \sigma_1}) \ldots t_n^{k_{n-1}}t_{n+1}^{k_n}) \ \overset{Eq.~\ref{sllem}}{=}\\
&&\\
 & = & \underset{A}{\underbrace{(q-1)\, \underset{j=0}{\overset{k_0-1}{\sum}}\, q^j\, tr\left(t^{p+j}t_1^{k_0-j}\, \tau_{2, n+1}^{k_{1, n}} \right)}}\ +\ \underset{B}{\underbrace{q^{k_0}\, tr\left(t^{p+k_0}\, \tau_{2, n+1}^{k_{1, n}}\cdot \sigma_1 \right)}},\\
\end{array}
\]

\noindent and

\[
\begin{array}{lcl}
tr(t^{p} t_1^{-k_0}\underline{\ldots t_n^{-k_{n-1}}t_{n+1}^{-k_n}\, \sigma_1^{-1}}) & = & tr(t^{p}\, (\underline{t_1^{-k_0}\cdot \sigma_1^{-1}}) \ldots t_n^{-k_{n-1}}t_{n+1}^{-k_n}) \ \overset{Eq.~\ref{sllem}}{=}\\
&&\\
 & = & \underset{A^{\prime}}{\underbrace{(q^{-1}-1)\, \underset{j=0}{\overset{-k_0+1}{\sum}}\, q^j\, tr\left(t^{p+j}t_1^{-k_0-j}\, \tau_{2, n+1}^{-k_{1, n}} \right)}}\ +\\
&&\\
& + & \underset{B^{\prime}}{\underbrace{q^{-k_0}\, tr\left(t^{p-k_0}\, \tau_{2, n+1}^{-k_{1, n}}\cdot \sigma_1^{-1} \right)}}\\
\end{array}
\]

Observe now that the terms $B$ and $B^{\prime}$ are ``symmetric'' and also that the term $t^{p+k_0}\, \tau_{2, n+1}^{k_{1, n}}\cdot \sigma_1$ in $B$ has a ``gap'' in the indices, and thus, it is of lower order than $\tau$ according to Definition~\ref{order}(b)(ii)($\alpha$). Note also that in $A$, the terms $t^{p+j}t_1^{k_0-j}\, \tau_{2, n+1}^{k_{1, n}}$ are ``symmetric'' with the corresponding terms in $A^{\prime}$ for all $j$, and that for $j\neq 0$, all terms are of lower order than $\tau$. Thus, from the induction hypothesis, we have that
\[
\begin{array}{lclr}
I\left( tr(B^{\prime})\right) & = &  tr(B)  &\\
&&&\\
I\left( tr(A^{\prime})\right) & = &  tr(A)  & {\rm for\ all}\ j\neq 0.\\
\end{array}
\]

For $j=0$ in $A$ we obtain the element $tr\left(t^{p}\, \tau_{1, n+1}^{k_{0, n}} \right)$ and for $j=0$ in $A^{\prime}$ we obtain $tr\left(t^{p}\, \tau_{1, n+1}^{-k_{0, n}} \right)$. These element's are ``symmetric'' and we have the following:

\[
\begin{array}{lcl}
tr\left(t^{p}\, \tau_{1, n+1}^{k_{0, n}} \right) & = & tr\left(t^{p}\, \tau_{1, n}^{k_{0, n-1}}\, \underline{t_{n+1}^{k_n}} \right)\ =\ \underset{A}{\underbrace{\underset{j=1}{\overset{k_n-1}{\sum}}\, q^{j-1}\, (q-1)\, tr\left(t^p t_{1, n}^{k_{0, n-1}}\, t_n^j\, t_{n+1}^{k_n-j}\cdot \sigma_{n+1} \right)}}\\
&&\\
& + & \underset{B}{\underbrace{q^{k_n}\, t^p\, \tau_{1, n}^{k_{0, n-1}}\, \sigma_{n+1}\, t_n^{k_n}\, \sigma_{n+1}}}\\
&&\\
{\rm and} &&\\
&&\\
tr\left(t^{p}\, \tau_{1, n+1}^{-k_{0, n}} \right) & = & tr\left(t^{p}\, \tau_{1, n}^{-k_{0, n-1}}\, \underline{t_{n+1}^{-k_n}} \right)\ =\ \underset{A^{\prime}}{\underbrace{\underset{j=1}{\overset{k_n-1}{\sum}}\, q^{j+1}\, (q^{-1}-1)\, tr\left(t^p t_{1, n}^{-k_{0, n-1}}\, t_n^j\, t_{n+1}^{-k_n+j}\cdot \sigma_{n+1}^{-1} \right)}}\\
&&\\
& = & \underset{B^{\prime}}{\underbrace{q^{-k_n}\, t^p\, \tau_{1, n}^{-k_{0, n-1}}\, \sigma_{n+1}^{-1}\, t_n^{-k_n}\, \sigma_{n+1}^{-1}}}\\
\end{array}
\]

We observe again that the terms $A$ and $A^{\prime}$ are ``symmetric'' and of lower order than $\tau, f(\tau)$ and thus, from the induction hypothesis we have that $I\left( A^{\prime}\right)\ =\ tr\left(A \right)$. For the terms $B, B^{\prime}$ we only have that the coefficients are ``symmetric''. The idea is to change the order of particular exponents that will lead to lower order terms in the resulting sum, when applying Theorem~\ref{exp}. We demonstrate this technique for the case $k_{n-1}<k_{n}$:

\[
\begin{array}{lcl}
B & = & q^{k_n}\, tr\left(t^p\, \tau_{1, n-1}^{k_{0, n-2}}\, \underline{t_n^{k_{n-1}}\, \sigma_{n+1}}\, t_n^{k_n}\, \sigma_{n+1}\right)\ \overset{Lem.~\ref{llem}}{=}\\
&&\\
& = & \underset{j=1}{\overset{k_{n-1}-1}{\sum}}\, tr\left(. . .\right) \ +\ q^{k_n-k_{n-1}+1}\, tr\left(\underline{t^p\, \tau_{1, n-1}^{k_{0, n-2}}\, \sigma_{n+1}^{-1}}\, t_{n+1}^{k_{n-1}}\, t_n^{k_n}\, \sigma_{n+1}\right)\ =\\
&&\\
& = & \underset{j=1}{\overset{k_{n-1}-1}{\sum}}\, tr\left(. . .\right) \ +\ q^{k_n-k_{n-1}+1}\, tr\left(t^p\, \tau_{1, n-1}^{k_{0, n-2}}\, t_{n+1}^{k_{n-1}}\, t_n^{k_n}\right)
\end{array}
\]

Similarly for $B^{\prime}$ we obtain:

\[
B^{\prime}\ \overset{Lem.~\ref{llem}}{=}\ \underset{j=1}{\overset{k_{n-1}-1}{\sum}}\, tr\left(. . .\right) \ +\ q^{-k_n+k_{n-1}-1}\, tr\left(t^p\, \tau_{1, n-1}^{-k_{0, n-2}}\, t_{n+1}^{-k_{n-1}}\, t_n^{-k_n}\right)
\]

Monomials in $t_i$'s in the sums in these relations are ``symmetric'' and of lower order than the initial monomial, and, assuming $k_{n-1}<k_{n}$, same is true for $t^p\, \tau_{1, n-1}^{k_{0, n-2}}\, t_{n+1}^{k_{n-1}}\, t_n^{k_n}$. The result follows from the induction step.
\end{proof}

Lemma~\ref{coef} and Propositions~\ref{trtau} and \ref{fl} are summarized in the following diagram:

\[
\begin{array}{ccccc}
tr\left(f(\tau_{0, n}^{k_{0, n}})\right) & \overset{bbm_{\mp}}{=} &  \frac{\lambda^{-1-\underset{i=0}{\overset{n}{\sum}}\, (i+1)k_i}}{z} & \cdot & tr(t^p \tau_{1, n+1}^{-k_0, n}g_1^{-1})\\
&&&&\\
f\, \uparrow \, Prop.~\ref{trtau} &  & I\, \uparrow\, Lem.~\ref{coef} & & I\, \uparrow\, Prop.~\ref{fl}\\
&&&&\\
tr(\tau_{0, n}^{k_{0, n}}) & \overset{bbm_{\pm}}{=} &\frac{\lambda^{\underset{i=0}{\overset{n}{\sum}}\, (i+1)k_i}}{z} & \cdot & tr(t^p \tau_{1, n+1}^{k_0, n}g_1)\\
\end{array}
\]

The proof of Corollary~\ref{cor1}, and thus, the proof of the main theorem, Theorem~\ref{mthm}, is  now concluded.

\section{Toward the HOMFLYPT skein module of $L(p,1)$}\label{slp1}

In this section we present some results concerning the full solution of the infinite system of equations (\ref{inft}), a solution of which corresponds to computing $\mathcal{S}(L(p, 1))$. We start by only considering equations obtained by the performance of braid band moves on elements in $\Lambda^{aug_+}$ and we prove that the infinite system of equations splits into self-contained subsystems. More precisely:

\begin{lemma}\label{trs}
Let $\tau \in \Lambda^{aug_+}_k$. Then $tr(\tau)\ =\ \underset{i}{\sum}\, f_i(q,z)\, s_{1,v}^{u_{1,v}}$, where $s_{1,v}^{u_{1,v}}\ :=\ s_1^{u_1}s_2^{u_2}\ldots s_{v}^{u_v}$, such that $u_i \in \mathbb{N}$ for all $i$ and $\underset{i=1}{\overset{v}{\sum}}\, i\cdot u_i\ =\ k$.
\end{lemma}

\begin{proof}
It derives directly from the change of basis matrix and the fourth rule of the trace.
\end{proof}

\begin{cor}
For $k\in \mathbb{N}$ we obtain an infinite self-contained system of equations by performing bbm's on elements in $\Lambda^{aug_+}_{(k)}$. That is, the system $(\ref{inft})$ splits into infinitely many self-contained subsystems of equations.
\end{cor}

\begin{remark}\rm
Note that if instead of considering elements in $\Lambda^{aug_+}_{(k)}$ and performing $bbm_1$'s, we considered elements in the set $\Lambda^{aug}_{(k)}, k\in \mathbb{Z}$, and perform $bbm_1$'s, then the infinitely many sub-systems of equations would again be self-contained, but they would be of infinite dimension, i.e. the number of equations and unknowns in each sub-system would be infinite. This is due to the fact that the exponents in the monomials in $t_i$'s in $\Lambda^{aug}_{(k)}$ are arbitrary (positive and negative). We call such monomials, monomials of ``mixed'' exponents and we deal with the equations obtained by performing bbm's on these elements in a sequel paper. For more details the reader is referred to \cite{D}.
\end{remark}

\subsection{Potential bases for the submodules $\frac{\Lambda^{aug_+}}{<bbm>}\ \&\ \frac{\Lambda^{aug_-}}{<bbm>}$}\label{gensetbas}

We now present some sub-systems with their solutions (without evaluating the coefficients) and we present generating sets (which also form potential bases) for $\frac{\Lambda^{aug_+}}{<bbm>}$ and $\frac{\Lambda^{aug_-}}{<bbm>}$. In particular, we consider elements in the subset of level $k\in \mathbb{N}$ of $\Lambda^{aug_+}$, $\Lambda^{aug_+}_{(k)}$, and we perform positive and negative bbm's on their first moving strand. We present some of the equations of the infinite system obtained that way:

\smallbreak

\begin{itemize}
\item[$\bullet$] For $k=0$ we have $1 \in \Lambda^{aug_+}_{(0)}$ and applying $bbm_{\pm 1}$'s we obtain:
$$1\overset{bbm_{\pm }}{\rightarrow}t^p\sigma_1^{\pm 1} \Rightarrow \color{red} 1:=s_{0}=s_{p}$$ 
\color{black}

\smallbreak

\item[$\bullet$] For $k=1$ we have $t \in \Lambda^{aug_+}_{(1)}$ and applying $bbm_{\pm 1}$'s we obtain:
$$t\overset{bbm_{\pm }}{\rightarrow}t^pt_1\sigma_1^{\pm 1} \Rightarrow \color{red} s_{p+1}=s_{1}\ \color{black}\&\ \color{red} s_ps_1=a_0s_1$$
\color{black}

\bigbreak

\item[$\bullet$] For $k=2$ we have $ t^2, \, tt_1 \in \Lambda^{aug_+}_{(2)}$ and applying $bbm_{\pm 1}$'s we obtain:
	\begin{itemize}
		\item[] $t^2 \overset{bbm}\rightarrow t^pt_1^{2}\sigma_1^{\pm 1}:$
		\[\left\{\begin{array}{ccl}
			s_2 & = & a_1 s_{p+2} + a_2 s_{p+1}s_1 + a_3s_ps_2 \\
			s_2 & = & b_1s_{p+2} + b_2 s_{p+1}s_1
		\end{array}\right. \]
		\item[] $tt_1 \overset{bbm}\rightarrow t^pt_1t_{2}\sigma_1^{\pm 1}:$ \[\left\{\begin{array}{ccl}
		s_2 + s_1^2& = & c_1 s_{p+2} + c_2 s_{p+1}s_1 \\
		s_2 +s_1^2 & = & d_1 s_{p+2} + d_2 s_{p+1}s_1 + d_3s_ps_2 + d_4s_p s_1^2
		\end{array}\right. \]
			\end{itemize}
and thus:\color{red}
\[\left\{\begin{array}{ccl}
			s_{p+2} & = & A_1 s_{2} + A_2 s_1^2 \\
			s_2s_p & = & B_1 s_{2} + B_2 s_1^2\\
            s_1s_{p+1} & = & C_1 s_{2} + C_2 s_1^2\\
            s_p s_1^2 & = & D_1 s_{2} + D_2 s_1^2\\
		\end{array}\right. \]
\bigbreak
\color{black}
\noindent where $a_i, A_i, b_i, B_i, c_i, C_i, d_i, D_i \in \mathbb{C},\ \forall\ i$.
\end{itemize}

\bigbreak

Observe now that for a fixed $k\in \mathbb{N}$, the set $\Lambda^{aug_+}_{(k)}$ has $\underset{i=0}{\overset{k-1}{\sum}}\, \bigl(\begin{smallmatrix}
k-1\\
i
\end{smallmatrix}\bigr) = 2^{k-1}$ elements and by performing a positive and a negative bbm on each element in $\Lambda^{aug_+}_{(k)}$, we obtain $2^k$ equations. We denote the subsystem obtained from elements in $\Lambda^{aug_+}_{(k)}$ by $[S_{(k)}]$. From Lemma~\ref{trs} we have that the unknowns in $[S_{(k)}]$ are of the form $s_1^{u_1}s_2^{u_2}\ldots s_{v}^{u_v}$, such that $u_i \in \mathbb{N}$ for all $i$ and $\underset{i=1}{\overset{v}{\sum}}\, i\cdot u_i\ =\ k$. Note also that the unknowns of the subsystem $[S_{(k)}]$ are related to the unknowns of the subsystems $[S_{(p+k)}], [S_{(2p+k)}], \ldots, [S_{(np+k)}],\ n\in \mathbb{N}$, by performing $bbm_1$'s, and thus, the solutions of the subsystems 
$$[S_{(0)}], [S_{(1)}], \ldots, [S_{(p-1)}]$$ 
\noindent produce a generating set of $\frac{\Lambda^{aug_+}}{<bbm>}$. Recall also that the unknown of the infinite system $s_m$, corresponds to the looping generator ${t_i^{\prime}}^m$ for any $i\in \mathbb{N}$, since $t_i^{\prime}$'s are conjugates.

\bigbreak

These lead to the following theorem:

\begin{thm}\label{gensetp}
The set
\begin{equation}\label{setp}
\left\{{t^{\prime}}^{k_0}{t^{\prime}}^{k_1} \ldots {t_n^{\prime}}^{k_n},\ {\rm where}\ n,\, k_i\in \mathbb{N}\ :\ 0\leq k_i\leq p-1 \right\}
\end{equation}
\noindent is a generating set of the module $\frac{\Lambda^{aug_+}}{<bbm_1>}$.
\end{thm}

\begin{remark}\rm
From Theorem~\ref{dlpl} we have that $\frac{\Lambda^{aug_+}}{<bbm_1>}\ =\ \frac{\Lambda^+}{<bbm_i>}$, and as explained in \cite{DLP, DL4}, $\frac{\Lambda^+}{<bbm_i>}\ =\ \frac{{\Lambda^{\prime}}^+}{<bbm_i>}$. Thus, the set in Eq.~\ref{setp} forms a basis for $\mathcal{B}^+/<bbm_i>$.
\end{remark}

\smallbreak

Moreover, results on the infinite system so far suggest that each subsystem admits unique solution and thus, $\frac{\Lambda^{aug_+}}{<bbm>}$ is torsion free, which suggest that the following conjecture is true:

\begin{conj}\label{con1}
The set
$$\left\{{t^{\prime}}^{k_0}{t^{\prime}}^{k_1} \ldots {t_n^{\prime}}^{k_n},\ {\rm where}\ n, \, k_i\in \mathbb{N}\ :\ 0\leq k_i\leq p-1 \right\}$$
\noindent forms a basis for $\frac{\Lambda^{aug_+}}{<bbm_1>}$.
\end{conj}

By Theorem~\ref{mthm} we have that the solution of the infinite system of equations $\frac{\Lambda^{aug_-}}{<bbm>}$ is derived from the solution of the infinite system of equations $\frac{\Lambda^{aug_+}}{<bbm>}$. Combined with Theorem~\ref{gensetp}, this leads to the following:

\begin{thm}\label{gensetn}
The set
$$\left\{{t^{\prime}}^{k_0}{t^{\prime}}^{k_1} \ldots {t_n^{\prime}}^{k_n},\ {\rm where}\ n,\, k_i\in \mathbb{Z}\backslash \mathbb{N}\ :\ 0\geq k_i\geq -p+1 \right\}$$
\noindent is a generating set of the module $\frac{\Lambda^{aug_-}}{<bbm>}$.
\end{thm}

And similarly to Conjecture~\ref{con1}, we have that:

\begin{conj}
The set 
$$\left\{{t^{\prime}}^{k_0}{t^{\prime}}^{k_1} \ldots {t_n^{\prime}}^{k_n},\ {\rm where}\ n,\, k_i\in \mathbb{N}\ :\ 0\geq k_i\geq -p+1 \right\}$$ \noindent forms a basis for $\frac{\Lambda^{aug_-}}{<bbm>}$.
\end{conj}

\subsection{Further Research}\label{further}

We now consider the subsystems of equations obtained by performing braid band moves on elements in $\Lambda^{aug_-}_{(k<0)}$. We use Theorem~\ref{mthm} and the equations obtained from elements in $\Lambda^{aug_+}$, presented above.

\smallbreak

\begin{itemize}
\item[$\bullet$] For $k=-1$ we have $t^{-1} \in \Lambda^{aug_-}_{(-1)}$ and: $t^{-1}\overset{bbm_{\pm }}{\rightarrow}t^pt_1^{-1}\sigma_1^{\pm 1} \Leftrightarrow \color{red} s_{p-1}=s_{-1} \ \color{black}\&\ \color{red} s_ps_{-1}=a^{\prime}_0s_1$ \color{black}.

\bigbreak

\item[$\bullet$] For the elements in $\Lambda^{aug_-}_{(-2)}$ we have:
	\begin{itemize}
		\item[] $t^{-2} \overset{bbm}\rightarrow t^pt_1^{-2}\sigma_1^{\pm 1}$
		\smallbreak
		\item[] $t^{-1}t_1^{-1} \overset{bbm}\rightarrow t^pt_1^{-1}t_{-2}\sigma_1^{\pm 1}$
and:\color{red}
\[\left\{\begin{array}{ccl}
			s_{p-2} & = & A_1^{\prime} s_{-2} + A_2^{\prime} s_{-1}^2 \\
			s_{-2}s_p & = & B_1^{\prime} s_{-2} + B_2^{\prime} s_{-1}^2\\
            s_{-1}s_{p-1} & = & C_1^{\prime} s_{-2} + C_2^{\prime} s_{-1}^2\\
            s_p s_{-1}^2 & = & D_1^{\prime} s_{-2} + D_2^{\prime} s_{-1}^2\\
		\end{array}\right. \]
\bigbreak
\color{black}
\noindent where $A_i^{\prime}, B_i^{\prime}, C_i^{\prime}, D_i^{\prime} \in \mathbb{C},\ \forall\ i$.
\end{itemize}
\end{itemize}

\bigbreak

Observe now for example that the unknown $s_{p-1}$ is equal to $s_{-1}$ and that the unknown $s_{p-2}$ can be written as a combination of $s_{-2}$ and $s_{-1}^2$. Thus, $s_{p-1}, s_{p-2}$ will not be in the basis of $\mathcal{S}\left( L(p,1) \right)$. The above examples of equations combined with the equations presented before for the system obtained from elements in $\Lambda^{aug_+}_{(k>0)}$, and results from \cite{D}, suggest that the following set forms a basis for $\mathcal{S}(L(p, 1))$: 

\[
\left\{{t^{\prime}}^{k_0}{t^{\prime}}^{k_1} \ldots {t_n^{\prime}}^{k_n},\ {\rm where}\ n,\, k_i\in \mathbb{Z}\ :\ -p/2 \leq k_i < p/2 \right\}
\]

It is worth mentioning that the same set was obtained and proved to be a basis for $\mathcal{S}(L(p, 1))$ in \cite{GM} using diagrammatic methods. In \cite{DL5} we work toward proving this result using braids and techniques described within this paper. We have reasons to believe that the braid technique can be successfully applied in order to compute skein modules of other more complicated c.c.o. 3-manifolds (such as the lens spaces $L(p, q), q>1$), where diagrammatic methods fail to do so.


\begin{thebibliography}{ABCD}


\bibitem[D]{D} {\sc I. Diamantis}, The HOMFLYPT skein module of the lens spaces $L(p,1)$ via braids, {\em PhD thesis},
National Technical University of Athens, 2015.

\bibitem[D2]{D2} {\sc I. Diamantis}, (2019) An Alternative Basis for the Kauffman Bracket Skein Module of the Solid Torus via Braids. In: Adams C. et al. (eds) Knots, Low-Dimensional Topology and Applications. KNOTS16 2016. Springer Proceedings in Mathematics \& Statistics, vol 284. Springer, Cham.

\bibitem[D3]{D3} {\sc I. Diamantis}, The Kauffman bracket skein module of the handlebody of genus 2 via braids, {\em J. Knot Theory and Ramifications}, {\bf 28}, No. 13, 1940020 (2019).

\bibitem[DL1]{DL1} {\sc I. Diamantis, S. Lambropoulou}, Braid equivalences in 3-manifolds 
with rational surgery description, {\em Topology and its Applications}, {\bf 194} (2015), 269-295.

\bibitem[DL2]{DL2} {\sc I. Diamantis, S. Lambropoulou}, A new basis for the HOMFLYPT skein module of the solid torus, {\em J. Pure Appl. Algebra} {\bf 220} Vol. 2 (2016), 577-605.

\bibitem[DL3]{DL3} {\sc I. Diamantis, S. Lambropoulou}, The braid approach to the HOMFLYPT skein module of the lens spaces $L(p, 1)$, Springer Proceedings in Mathematics and Statistics (PROMS), {\em Algebraic Modeling of Topological and Computational Structures and Application}, (2017).

\bibitem[DL4]{DL4} {\sc I. Diamantis, S. Lambropoulou}, An important step for the computation of the HOMFLYPT skein module of the lens spaces $L(p,1)$ via braids, {\em J. Knot Theory and Ramifications}, {\bf 28}, No. 11, 1940007 (2019).

\bibitem[DL5]{DL5} {\sc I. Diamantis, S. Lambropoulou}, The HOMFLYPT skein module of the lens spaces $L(p,1)$ via braids, in preparation.

\bibitem[DLP]{DLP} {\sc I. Diamantis, S. Lambropoulou, J. H. Przytycki}, Topological steps on the HOMFLYPT skein module of the lens spaces $L(p,1)$ via braids, {\em J. Knot Theory and Ramifications}, {\bf 25}, No. 14, (2016).

\bibitem[FYHLMO]{FYHLMO} {\sc P. Freyd, D. Yetter, J. Hoste, W. B. R. Lickorish, K. Millett, A. Ocneanu}, A new polynomial invariant of knots and links, {\em Bull. Amer. Math. Soc.}, {\bf 12} (1985), 239-249.

\bibitem[GM]{GM} {\sc B. Gabrov\v sek, M. Mroczkowski}, The Homlypt skein module of the lens spaces $L(p,1)$, {\em Topology and its Applications}, {\bf 175} (2014), 72-80.

\bibitem[HK]{HK} {\sc J.~Hoste, M.~Kidwell}, Dichromatic link invariants, {\it Trans. Amer. Math. Soc.} {\bf 321} (1990), No. 1, 197-229.

\bibitem[HP]{HP} {\sc J.Hoste, J.H.Przytycki}, A survey of skein modules of 3-manifolds. {\it Knots 90 (Osaka, 1990)}, de Gruyter, Berlin, (1992) 363â?-379.

\bibitem[Jo]{Jo}{\sc V. F. R. Jones}, A polynomial invariant for links via Neumann algebras, {\it Bull. Amer. Math. Soc.} {129},
(1985) 103-112.

\bibitem[La1]{La1}{\sc S. Lambropoulou}, Knot theory related to generalized and cyclotomic Hecke algebras of type {\it B}, {\it J. Knot Theory and its Ramifications} {\bf 8},
              No. 5, (1999) 621-658.

\bibitem[La2]{La2} {\sc S. Lambropoulou}, Solid torus links and Hecke algebras of B-type, {\it Quantum Topology}; D.N. Yetter Ed.; World Scientific Press, (1994), 225-245.

\bibitem[LR1]{LR1} {\sc S. Lambropoulou, C.P. Rourke} (2006), Markov's theorem in $3$-manifolds, \emph{Topology and its Applications} {\bf 78},
(1997) 95-122.

\bibitem[LR2]{LR2} {\sc S. Lambropoulou, C. P. Rourke}, Algebraic Markov equivalence for links in $3$-manifolds, {\em Compositio Math.} {\bf 142} (2006) 1039-1062.

\bibitem[P]{P} {\sc J.~Przytycki}, Skein modules of 3-manifolds, {\it Bull. Pol. Acad. Sci.: Math.}, {\bf 39, 1-2} (1991), 91-100.

\bibitem[PT]{PT} {\sc J. H. Przytycki, P. Traczyk}, Invariants of links of Conway type, \emph{Kobe J. Math.} {\bf 4} (1987), 115-139.

\bibitem[Tu]{Tu} {\sc V.G.~Turaev}, The Conway and Kauffman modules of the solid torus,  {\it Zap. Nauchn. Sem. Lomi} {\bf 167} (1988), 79--89. English translation: {\it J. Soviet Math.} (1990), 2799-2805.

\end{thebibliography}
\end{document}